\documentclass[12pt]{amsart}
\usepackage[top=1in,left=1in,bottom=1in,right=1in]{geometry}
\usepackage{amsmath}
\usepackage{amssymb}
\usepackage{amsthm}
\usepackage{tikz}
\usetikzlibrary{arrows}
\usepackage{soul}
\usepackage{wrapfig}
\usepackage{graphicx}
\usepackage{epstopdf}
\usepackage{tikz-cd}
\usepackage{url}
\usepackage{scrextend}
\usepackage{colordvi}
\usepackage{color}
\usepackage{verbatim}

\usepackage{amsmath}
\usepackage{amsfonts}
\usepackage{colordvi}
\usepackage{color}
\usepackage{amsthm}
\usepackage{url}
\usepackage{graphicx}
\usepackage{epstopdf}
\usepackage{amssymb}

\newtheorem{innercustomthm}{Theorem}
\newenvironment{restatetheorem}[1]
  {\renewcommand\theinnercustomthm{#1}\innercustomthm}
  {\endinnercustomthm}
\newtheorem{theorem}{Theorem}[section]
\newtheorem{lemma}[theorem]{Lemma}

\newtheorem{question}[theorem]{Question}
\newtheorem{proposition}[theorem]{Proposition}

\newtheorem{corollary}[theorem]{Corollary}
\theoremstyle{definition}
\newtheorem{innercustomdef}{Definition}
\newenvironment{restatedefinition}[1]
  {\renewcommand\theinnercustomdef{#1}\innercustomdef}
  {\endinnercustomdef}

\newtheorem{definition}[theorem]{Definition}
\theoremstyle{remark}
\newtheorem{remark}[theorem]{Remark}
\newtheorem{example}[theorem]{Example}

\newcommand{\C}{\mathcal C}
\newcommand{\D}{\mathcal D}
\newcommand{\E}{\mathcal E}

\newcommand{\T}{\mathcal T}

\newcommand{\U}{\mathcal U} 
\newcommand{\V}{\mathcal V}

\renewcommand{\P}{\mathcal P}
\newcommand{\A}{\mathcal A}

\newcommand{\R}{\mathbb R}
 
\renewcommand{\S}{\mathbb S}

\DeclareMathOperator{\conv}{conv}
\DeclareMathOperator{\cl}{cl}
\DeclareMathOperator{\spann}{span}

\DeclareMathOperator{\Tk}{Tk}

\DeclareMathOperator{\odim}{odim}
\DeclareMathOperator{\cdim}{cdim}
\DeclareMathOperator{\code}{code}
\DeclareMathOperator{\interior}{int}

\DeclareMathOperator{\trim}{trim}

\newcommand{\od}{:=}

\newcommand{\Code}{\mathbf{Code}}

\newcommand{\ParCode}{\mathbf{P}_\Code}
\renewcommand{\S}{\mathcal S}
\newcommand{\edit}[1]{#1}

\begin{document}
\title{Embedding Dimension Phenomena in Intersection Complete Codes}

\author{R. Amzi Jeffs}

\begin{abstract}
Two tantalizing invariants of a combinatorial code $\C\subseteq 2^{[n]}$ are $\cdim(\C)$ and $\odim(\C)$, the smallest dimension in which $\C$ can be realized by convex closed or open sets, respectively. Cruz, Giusti, Itskov, and Kronholm showed that for intersection complete codes $\C$ with $m+1$ maximal codewords, $\odim(\C)$ and $\cdim(\C)$ are both bounded above by $\max\{2,m\}$. Results of Lienkaemper, Shiu, and Woodstock imply that $\odim$ and $\cdim$ may differ, even for intersection complete codes. We add to \edit{the literature on open and closed embedding dimensions of intersection complete codes} with the following results:\begin{itemize}
\item If $\C$ is a simplicial complex, then $\cdim(\C) = \odim(\C)$,
\item If $\C$ is intersection complete, then $\cdim(\C)\le \odim(\C)$,
\item If $\C\subseteq 2^{[n]}$ is intersection complete with $n\ge 2$, then $\cdim(\C) \le \min \{2d+1, n-1\}$, where $d$ is the dimension of the simplicial complex of $\C$, and 
\item For each simplicial complex $\Delta\subseteq 2^{[n]}$ with $m\ge 2$ facets, the code $\S_\Delta$ \edit{$\od (\Delta \ast (n+1)) \cup \{[n]\}$} $\subseteq 2^{[n+1]}$ is intersection complete, has $m+1$ maximal codewords, and satisfies $\odim(\S_\Delta)=m$. In particular, for each $n\ge 3$ there exists an intersection complete code $\C\subseteq 2^{[n]}$ with $\odim(\C) = \binom{n-1}{\lfloor (n-1)/2\rfloor}$.  
\end{itemize}
A key tool in our work is the study of sunflowers: arrangements of convex open sets in which the sets simultaneously meet in a central region, and nowhere else. We use Tverberg's theorem to study the structure of ``$k$-flexible" sunflowers, and consequently obtain new lower bounds on $\odim(\C)$ for intersection complete codes $\C$. \end{abstract}

\thanks{Jeffs' research is supported by a graduate fellowship from NSF grant DGE-1761124}
\date{\today.Department of Mathematics.  University of Washington, Seattle, Wa 98195}
\maketitle

\section{Introduction}\label{sec:intro}

In \cite{neuralring13}, Curto, Itskov, Veliz-Cuba, and Youngs introduced \emph{convex codes} to mathematically model stimulus reconstruction from neural data, particularly in the context of hippocampal place cells. Classifying and understanding convex codes have been active areas of recent mathematical research, bringing together tools and perspectives from topology \cite{local15, connected}, algebra \cite{grobner, polarization, factorcomplex}, and discrete geometry \cite{undecidability, polarcomplex,sunflowers,CUR, obstructions}. A complete classification of convex codes is far out of reach for the moment, but progress can yield new techniques for analyzing neural data, as well as a deeper understanding of the mathematical theory of convex sets. In this paper, we give new bounds on the open and closed embedding dimensions of intersection complete codes, and families of examples where these bounds are tight. In particular, we provide infinite families of intersection complete codes for which open embedding dimension grows exponentially in the number of neurons, while closed embedding dimension grows only linearly.

Before stating our results we recall some definitions and frame our main questions of study. A convex code (see Definition \ref{def:convexcode} below) is a special case of a \emph{combinatorial code}, which is a collection of subsets of $[n]\od \{1,\ldots, n\}$. Due to the biological motivation behind our work, we think of the elements of $[n]$ as \emph{neurons}, and each element of a code as recording a set of neurons which fired together in a small window of time.

The elements of a code are called \emph{codewords}, and for concision we often omit braces and commas when writing codewords. For example, we may write 124 instead of $\{1,2,4\}$. The \emph{weight} of a codeword is simply the number of neurons it contains. We will often think of a code as a partially ordered set under containment---for example, we may speak of maximal codewords, which are not properly contained in any other codeword. When writing down a specific code, we will bold the maximal codewords.

Codes can arise abstractly when one wishes to describe how a certain collection of sets covers a space, as follows. Let $X$ be a set and $\U = \{U_1,\ldots, U_n\}$ a collection of subsets of $X$. One may form the \emph{code of $\U$ in $X$}, a combinatorial code whose codewords describe how the $U_i$ intersect and cover one another: \[
\code(\U, X) \od \bigg\{\sigma\subseteq [n]\ \bigg|\ \bigcap_{i\in\sigma} U_i \setminus \bigcup_{j\in[n]\setminus\sigma} U_j\neq \emptyset\bigg\}.
\]
The region $\bigcap_{i\in\sigma} U_i \setminus \bigcup_{j\in[n]\setminus\sigma} U_j$ is called the \emph{atom} of $\sigma$, and denoted by $\A_\U^\sigma$. The space $X$ is called the \emph{ambient space} or \emph{stimulus space}, and the $U_i$ are called \emph{receptive fields} or \emph{firing regions}. Note that the receptive fields are indexed by neurons.  If $\C = \code(\U, X)$, then the collection $\U$ is called a \emph{realization} of $\C$ in $X$. For concision, we will write $U_\sigma$ for $\bigcap_{i\in\sigma} U_i$ (and similarly define $V_\sigma$ when working with various $V_i$), and adopt the convention that $U_\emptyset = X$.

Unless otherwise specified, throughout this paper the ambient space will be $\R^d$, and the $U_i$ will be (possibly empty) convex sets that  are either all open, or all closed. We will write $\code(\U)$ instead of $\code(\U, \R^d)$ when the ambient dimension is clear. We will also adopt the usual convention in the study of convex codes that $\emptyset$ is contained in all codes, i.e. that there is always a point in the ambient space not covered by any $U_i$. In our examples we illustrate open sets with a solid border so that our figures are clean and readable. The accompanying text will always specify whether we are regarding the sets in question as open or closed. Finally, we will always assume $n\ge 1$ (i.e. that we are not working with an empty set of neurons).

\begin{definition}\label{def:convexcode}
A code $\C\subseteq 2^{[n]}$ is called an \emph{open convex} code if it has a realization consisting of convex open sets in $\R^d$. Similarly, $\C$ is called \emph{closed convex} if it has a realization consisting of closed convex sets in $\R^d$. 
\end{definition}

\begin{example}\label{ex:realization}
The figure below shows a realization in $\R^2$ of the (open/closed) convex code $\C = \{\mathbf{123}, 12, 23, 2, 3, \emptyset\}$. The atom $\A_\U^{23}$ is highlighted in grey.
\[
\includegraphics[width=22em]{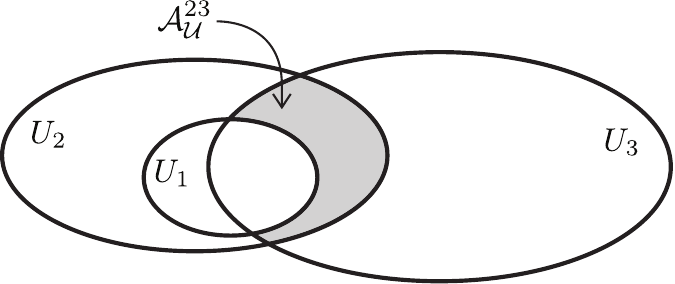}
\]
\end{example}

In the neuroscientific context mentioned above, open convex codes are of greater interest than closed convex codes, since receptive fields have been experimentally observed to be full-dimensional (see \cite[Figure 1]{local15} for example). However, we will also study closed convex codes in this paper to build on the work of \cite{openclosed}, and to contrast their behavior to that of open convex codes. Moreover, it is of broad mathematical interest to develop our understanding of closed convex sets and their intersection patterns, since they are ubiquitous in fields such as optimization and discrete geometry (see for example \cite[Section 3]{hellytoday}, \cite[Chapters 6 and 8]{matousek}, and \cite[Chapter 8]{schrijver}). 

The study of convex codes asks two main questions. First, given a code $\C\subseteq 2^{[n]}$, \textbf{when can we find a (closed or open) convex realization of $\mathbf{\C}$?} Second, if we can find a realization, \textbf{what is the smallest dimension in which we can do so?} Formally, we wish to investigate the open and closed embedding dimensions of combinatorial codes, described below.

\begin{definition}\label{def:odim}\label{def:cdim}
Let $\C\subseteq 2^{[n]}$ be a code. The \emph{open embedding dimension} of $\C$, denoted by $\odim(\C)$, is the smallest $d$ so that $\C$ has a realization in $\R^d$ consisting of  convex open sets, or $\infty$ if no convex open realization exists. Similarly, the \emph{closed embedding dimension}, denoted by $\cdim(\C)$, is the smallest dimension in which $\C$ has a closed convex realization, or $\infty$ if none exist. 
\end{definition}

\begin{remark}
The study of open and closed embedding dimension generalizes the the study of $d$-representable simplicial complexes, a classical topic in discrete geometry (see \cite{tancer} for a survey). A simplicial complex $\Delta$ is said to be \emph{$d$-representable} if one can find a collection $\U$ of convex sets in $\R^d$ such that for each $\sigma\subseteq [n]$ we have $\sigma\in \Delta$ if and only if $U_\sigma \neq \emptyset$ (equivalently, $\Delta$ is the smallest simplicial complex containing $\code(\U)$). In this context the requirement that all sets in $\U$ be open or closed is immaterial; one may work with open or closed convex sets interchangeably. Thus the differences that can arise between open and closed embedding dimensions of codes (which we will see starkly in Corollary \ref{cor:exponential}) provide interesting evidence that the study of convex codes significantly generalizes the study of $d$-representable complexes.
\end{remark}

Note that the convex realization in Example \ref{ex:realization} is not minimal with respect to dimension, since we could flatten the $U_i$ into (closed or open) intervals to obtain a realization in $\R^1$. Thus $\odim(\C) = \cdim(\C) = 1$ for $\C = \{\mathbf{123}, 12, 23, 2, 3, \emptyset\}$.

In this paper we will study codes that are \emph{intersection complete}: the intersection of any two codewords is again a codeword. An important result of \cite{openclosed} is that intersection complete codes are always open and closed convex. More specifically, we have the following: \begin{theorem}[Special case of Theorem 1.2 of \cite{openclosed}]\label{thm:openclosed}
Let $\C\subseteq 2^{[n]}$ be an intersection complete code with $m+1$ maximal codewords. Then \[
\max\{\odim(\C), \cdim(\C)\} \le \max\{2, m\}.
\]
\end{theorem} We will prove a new bound on closed embedding dimension which improves the bound in Theorem \ref{thm:openclosed} for many intersection complete codes (see Theorem \ref{thm:cdimlinear}). In contrast to the closed case, we will show that the bound $\odim(\C)\le \max\{2,m\}$ can be tight for any choice of $m\ge 2$ (see Theorem \ref{thm:SDelta}). Except where stated otherwise, every code we work with in this paper is intersection complete.

A special case of intersection complete codes is that of a simplicial complex. For simplicial complexes, open and closed embedding dimensions are equal. Although this result is well known among the neural codes community, we are not aware of any written proofs. We provide one below.

\begin{theorem}\label{thm:complexes}
Let $\C\subseteq 2^{[n]}$ be a simplicial complex. Then $\cdim(\C) = \odim(\C)$. 
\end{theorem}

\begin{proof}
In Theorem \ref{thm:cdimleodim}, we will show that $\cdim(\C)\le \odim(\C)$. Thus we just need to prove that $\odim(\C)\le\cdim(\C)$. Let $\V = \{V_1,\ldots, V_n\}$ be a closed realization of $\C$ in $\R^{\cdim(\C)}$. By intersecting all the $V_i$ with a sufficiently large closed ball, we may assume that they are bounded, and hence compact. For each nonempty codeword $c\in\C$, choose a point $p_c\in \A_\V^c$. By compactness, each $p_c$ has positive distance to any set $V_i$ that does not contain it. Likewise, any choice of $\sigma,\tau\subseteq [n]$ for which $V_\sigma$ and $V_\tau$ are disjoint, the sets $V_\sigma$ and $V_\tau$ must have positive distance between one another. Thus we may choose $\varepsilon$ such that replacing the $V_i$ by their Minkowski sums with an open $\varepsilon$-ball neither causes any $V_i$ to cover some $p_c$ it did not before, nor causes disjoint $V_\sigma$ and $V_\tau$ to intersect. This creates a collection of convex open sets whose code contains all the codewords of $\C$, and no new maximal codewords. Since $\C$ is a simplicial complex, this is exactly a convex open realization of $\C$. 
\end{proof}

\begin{example}
Consider the code $\C = \{\mathbf{123},\mathbf{34}, 12, 13, 23, 1,2,3,4,\emptyset\}$, and note that $\C$ is a simplicial complex. The lefthand side of the figure below shows a realization of $\C$ in $\R^2$ with closed convex sets, as well as possible choices of points $p_c$ for $c\in \C$ as used in the proof above. The righthand side shows the open realization given in the proof above, which results from adding a small $\varepsilon$-ball to each $V_i$. 

\[
\includegraphics[width=34em]{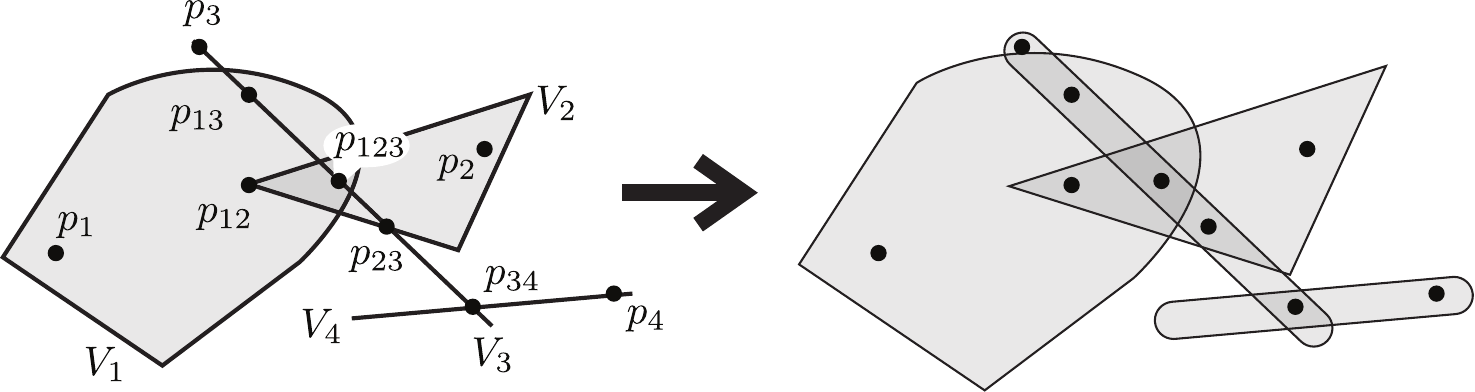}
\]
\end{example}

Can the above techniques be extended to realizations of codes that are not simplicial complexes? The answer in general is no, even for intersection complete codes, a fact which was first observed implicitly in the results of \cite{sunflowers, obstructions}. Corollary \ref{cor:exponential} will yield a plethora of examples of intersection complete codes on $n$ neurons that are closed convex in $\R^{n-1}$, but not open convex in $\R^{n-1}$. For such codes, adding an $\varepsilon$-ball to sets in a closed realization in $\R^{n-1}$ will always fail to produce an open realization.

The following theorems are the main contributions of this work, and give us a handle on how open and closed dimension behave for intersection complete codes. 

\begin{theorem}\label{thm:cdimleodim}
Let $\C\subseteq 2^{[n]}$ be an intersection complete code. Then $\cdim(\C)\le \odim(\C)$.
\end{theorem}

It is known that this inequality may be strict for intersection complete codes $\C\subseteq 2^{[n]}$, as mentioned above. In fact, the gap may be quite large: Theorem \ref{thm:cdimlinear} implies that $\cdim(\C)\le n-1$, while Corollary \ref{cor:exponential} says that $\odim(\C)$ \textbf{may be exponential in $n$}. 

\begin{theorem}\label{thm:cdimlinear}
Let $\C\subseteq 2^{[n]}$ be an intersection complete code {with $n\ge 2$}, and $d$ be one less than the weight of the largest codeword in $\C$ (i.e. $d=\dim(\Delta(\C))$). Then $\cdim(\C)\le \min \{2d+1, n-1\}$.
\end{theorem}

This bound is known to be tight. For every $d\ge 0$,  \cite{2dplus1} describes a $d$-dimensional simplicial complex on $n$ vertices whose closed embedding dimension is exactly $2d+1$ (which, in the family given, is the same as $\min\{2d+1, n-1\}$). {Our proof of Theorem \ref{thm:cdimlinear} is inspired by a construction of Wegner and Perel\textquotesingle man which was originally used to show that every $d$-dimensional simplicial complex is $(2d+1)$-representable (see \cite[Theorem 3.1]{tancer}). }

Interestingly, the bound {in Theorem \ref{thm:cdimlinear}} does not hold for $\odim(\C)$. {In fact,} Theorem \ref{thm:SDelta} below gives us a way to construct numerous examples of intersection complete codes for which $\odim(\C) \gg \min \{2d+1, n-1\}$.

\begin{definition}\label{def:SDelta}
Let $\Delta\subseteq 2^{[n]}$ be a simplicial complex. Define $\S_\Delta\subseteq 2^{[n+1]}$ to be the code \[
\S_\Delta \od \left(\Delta \ast (n+1)\right)\cup\{[n]\},
\]
where $\Delta\ast (n+1)$ denotes the cone over $\Delta$ with apex $n+1$. 
\end{definition}

\begin{theorem}\label{thm:SDelta}
Let $\Delta\subseteq 2^{[n]}$ be a simplicial complex with $m\ge 2$ facets. Then $\S_\Delta$ is an intersection complete code with $m+1$ maximal codewords, and $\odim(\S_\Delta) = m$.
\end{theorem}

A key tool in proving Theorem \ref{thm:SDelta} is an application of a ``sunflower theorem" that we proved in \cite{sunflowers}. In this paper, we will generalize this theorem to ``$k$-flexible" sunflowers of convex open sets, defined formally below. These are collections of convex open sets which have a common intersection, but {no more than $k$ of which} overlap outside of this common intersection.

\begin{definition}\label{def:flexiblesunflower}
Let $\U = \{U_1,\ldots, U_n\}$ be a collection of convex sets in $\R^d$ and let $\C = \code(\U)$. The collection $\U$ is called a \emph{$k$-flexible sunflower} if $[n]\in\C$, and all other codewords have weight at most $k$. The $U_i$ are called \emph{petals} and $U_{[n]}$ is called the \emph{center} of $\U$. \end{definition}

The following theorem tells us that if a $k$-flexible sunflower $\U$ in $\R^d$ has ``enough" petals, then sampling a point from each petal and taking the convex hull always {covers} a point in the center of $\U$. Our proof of this theorem is given in Section \ref{sec:flexible} and relies on an application of Tverberg's theorem.

\begin{theorem}\label{thm:flexible}
Let $\U = \{U_1,\ldots, U_n\}$ be an open $k$-flexible sunflower in $\R^d$. Suppose that $n\ge dk+1$, and for each $i\in[n]$ let $p_i\in U_i$. Then $\conv\{p_1,\ldots, p_n\}$ contains a point in the center of $\U$. Moreover, if $d\ge 2$ this result may fail when $n < dk+1$. 
\end{theorem}

By considering a set of line segments in $\R^2$ which meet at a point, one can see that this result does not hold for closed convex sets.

\begin{example}
Consider the {open} $2$-flexible sunflower $\{U_1,U_2,U_3,U_4,U_5\}$ in $\R^2$ below. The center of this sunflower is the unit square highlighted in {dark} gray. Note that $d=2$, $k=2$, and $n=5$. Thus $n \ge dk+1$, and so Theorem \ref{thm:flexible} applies. Indeed, any choice of $p_1\in U_1,\ldots, p_5\in U_5$ has the property that $\conv\{p_1,p_2,p_3,p_4,p_5\}$ intersects the center of the sunflower. One choice of such points is shown below.

\[
\includegraphics[width=18em]{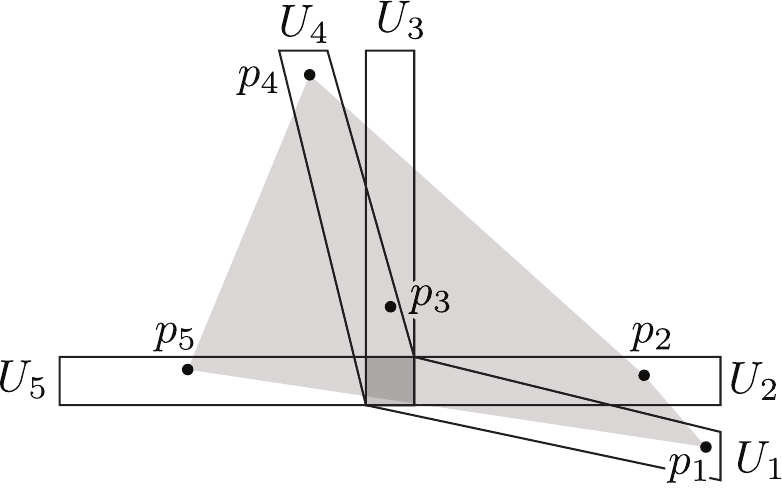}
\]

\edit{Observe} that deleting $U_5$ yields a $2$-flexible sunflower in $\R^2$ for which the conclusion of Theorem \ref{thm:flexible} does not hold: the set $\conv\{p_1, p_2, p_3,p_4\}$ does not intersect the center of $\{U_1, U_2, U_3, U_4\}$. 
\end{example}

In Section \ref{sec:background} we will recall some relevant background material. The subsequent sections are devoted to proving the theorems stated above, with one self-contained section per theorem. An exception to this is Section \ref{sec:sunflowercodeversion}, which \edit{provides important supporting results and context for} Section \ref{sec:SDelta}.

Section \ref{sec:Tn} describes a new family $\T_n$ of intersection complete codes, and initiates the study of their open embedding dimensions. The codes $\T_n$ are related to sunflowers, but the theorems that we prove regarding sunflowers are not sufficient to precisely determine $\odim(\T_n)$. 

Section \ref{sec:pcode} provides a unifying capstone to our results. \edit{We examine the families of codes from Sections \ref{sec:sunflowercodeversion}, \ref{sec:SDelta}, and \ref{sec:Tn} in the context of} a partially ordered set $\ParCode$ consisting of all neural codes, which was first introduced in \cite{morphisms}. We show that \edit{the bound} on open embedding dimension \edit{from Theorem \ref{thm:SDelta}} can be proven combinatorially using this partial order. We also generalize Definition \ref{def:SDelta}, and apply Theorem \ref{thm:flexible} to prove a generalization of Theorem \ref{thm:SDelta}, viewing these results through the lens of $\ParCode$.

\section{Background and Preliminaries}\label{sec:background}

Throughout this paper we will assume familiarity with standard concepts in topology and convex geometry; for example the interior, closure, and boundary of a set in $\R^d$, convex hulls, hyperplanes, and halfspaces \edit{(see} \cite[\edit{Chapter 1}]{matousek}\edit{)}. Recall that each hyperplane $H$ in $\R^d$ may be given an orientation, so that we can speak of the (open) halfspaces $H^>$ and $H^<$ consisting of points lying on the positive and negative sides of $H$, respectively. We will also use $H^\ge$ and $H^\le$ to denote the (closed) non-negative and non-positive respective halfspaces associated to $H$. For any convex set $U\subseteq \R^d$ and any boundary point $p$ of $U$, one can find a \emph{supporting hyperplane} through $p$: an oriented hyperplane $H$ containing $p$ with $U\subseteq H^\ge$. 

Below, we provide additional background on convex codes, simplicial complexes, and polytopes.

\subsection{Convex Codes}
In Section \ref{sec:intro} we gave a brief overview of the theory of convex neural codes. We will need one additional concept related to neural codes, described below.

\begin{definition}\label{def:trunk}
Let $\C\subseteq 2^{[n]}$ be a code, and let $\sigma\subseteq [n]$. The \emph{trunk} of $\sigma$ in $\C$ is \[
\Tk_\C(\sigma) \od \{c\in \C\mid \sigma\subseteq c\}.
\]
A subset of $\C$ is called a trunk if it is empty, or equal to $\Tk_\C(\sigma)$ for some $\sigma\subseteq [n]$. When $\sigma = \{i\}$ we will call $\Tk_\C(\sigma)$ a \emph{simple} trunk, and denote it $\Tk_\C(i)$.
\end{definition}

We introduced trunks in \cite{morphisms} and used them (and a consequent notion of morphism) to define a convenient partial order on neural codes, in which convex codes form a down-set. We will make use of this partial order to contextualize our results in Section \ref{sec:pcode}. 

It is worth briefly justifying our requirement that realizations consist of all closed or all open sets. As mentioned in Section \ref{sec:intro}, openness is a natural requirement from the perspective of neuroscience, in which receptive fields are full-dimensional and do not terminate in sharp boundaries. From a mathematical perspective, requiring closed or open sets is also natural, so that we may think of the receptive fields as a collection of closed or open sets covering some topological subspace of $\R^d$.  A further reason to place topological constraints on the sets in our realizations is the following: in \cite{allcodesconvex}, it was shown that every code has a realization consisting of convex sets (possibly neither open nor closed). Thus topological constraints are imperative to make the overall question of classifying convex codes meaningful.

\subsection{Simplicial Complexes}
For our purposes, an \emph{(abstract) simplicial complex} is just a code that is closed under taking subsets (i.e. a subset of a codeword is again a codeword). If $\Delta\subseteq 2^{[n]}$ is a simplicial complex, the maximal codewords may be called \emph{facets}, the codewords called \emph{faces}, and elements of $[n]$ called \emph{vertices}. Observe that every simplicial complex is uniquely specified by its facets together with the vertex set $[n]$. In contrast to the usual theory of simplicial complexes, we allow the case in which $i$ is a vertex but $\{i\}\notin \C$. 

The \emph{dimension} of a simplicial complex $\Delta$, denoted \edit{by} $\dim(\Delta)$, is one less than the size of the largest face in $\Delta$.
If $\Delta\subseteq 2^{[n]}$ is a simplicial complex, and $m > n$, the \emph{cone over $\Delta$ with apex $m$} is the simplicial complex \[
\Delta \ast m \od \{\sigma \subseteq [m] \mid \sigma\setminus \{m\} \in \Delta\}.
\]
That is, $\Delta\ast m$ is the simpicial complex whose facets are the facets of $\Delta$ with $m$ added to them. 
Finally, for any code $\C\subseteq 2^{[n]}$, the \emph{simplicial complex of $\C$}, denoted \edit{by} $\Delta(\C)$, is the smallest simplicial complex containing $\C$. 

\subsection{Polytopes and Polytopal Complexes}

A \emph{polytope} is the convex hull of a finite set of points in $\R^d$, or equivalently a bounded intersection of finitely many closed halfspaces. The \emph{dimension} of a polytope is the dimension of its affine hull. The \emph{(proper) faces} of a $d$-dimensional polytope in $\R^d$ are its intersections with supporting hyperplanes; faces consisting of a single point are called \emph{vertices}, and maximal faces are called \emph{facets}. We will also consider the empty set to be a proper face of any polytope $P$, and its associated supporting hyperplane to be any hyperplane that does not intersect $P$.

One can partially order the faces of a polytope by inclusion to form its \emph{face poset}. Two polytopes are called \emph{combinatorially equivalent} if their face posets are isomorphic. Every polytope $P\subseteq \R^d$ admits a \emph{dual polytope} $P^*\subseteq \R^d$, which has the property that the face poset of $P^*$ is isomorphic to the dual of the face poset of $P$ (i.e. one obtains the face poset of $P^*$ by turning the face poset of $P$ upside down). 

 A polytope is called \emph{$d$-neighborly} if the convex hull of any $d$ of its vertices is a face. Conveniently, $d$-neighborly polytopes with an arbitrarily large number of vertices can always be found in $\R^{2d}$ (e.g. the cyclic polytope; see \cite[Corollary 0.8]{ziegler-polytopesbook}). 
 
 A \emph{polytopal complex} in $\R^d$ is a finite set of polytopes $\P$ with the properties that (i) if $P\in \P$, then any face of $P$ is also in $\P$, and (ii) the intersection of two polytopes $P_1,P_2\in \P$ is a face of both $P_1$ and $P_2$. Polytopes in $\P$ are called \emph{faces}.
 
  Each polytopal complex $\mathcal P$ has a \emph{face poset}, consisting of all faces in $\P$ partially ordered by containment. Two polytopal complexes are called \emph{combinatorially equivalent} if their face posets are isomorphic. Maximal faces in $\P$ are called \emph{facets}, and if all facets have the same dimension then $\P$ is called \emph{pure}. Finally, we say that a polytopal complex $\P$ in $\R^d$ is \emph{full-dimensional} if it has a facet of dimension $d$. 
 
Given a $d$-dimensional polytope $P\subseteq \R^d$ and a facet $F$ of $P$, one can form a pure, full-dimensional polytopal complex in $\R^{d-1}$ called the \emph{Schlegel diagram of $P$ based at $F$}. Roughly, one does this by ``looking through" the facet $F$ to project all other faces of $P$ into $\R^{d-1}$. The key fact about Schlegel diagrams that we will need is the following: as a polytopal complex, the Schlegel diagram is combinatorially equivalent to the complex of all proper faces of $P$, but with $F$ removed. 
For further background on polytopes and polyhedral complexes, we refer the reader to \cite[Chapter 5]{ziegler-polytopesbook}.

\section{Closed Embedding Dimension is Bounded by Open Embedding Dimension}\label{sec:cdimleodim}

To begin our investigation, we recall a useful characterization of intersection complete codes in terms of their realizations. This fact has been observed before in various forms, for example \cite[Theorem 1.9]{signatures}. 

\begin{proposition}\label{prop:singlecover}
A code $\C\subseteq 2^{[n]}$ is intersection complete if and only if the following holds: for all $\sigma\in \Delta(\C)\setminus\C$ and all (possibly non-convex) realizations $\U = \{U_1,\ldots U_n\}$ of $\C$ there is some $i\in [n]\setminus \sigma$ with $U_\sigma\subseteq U_i$.
\end{proposition}
\begin{proof}
First suppose that $\C$ is intersection complete, and has a realization $\U = \{U_1,\ldots,U_n\}$. Let $\sigma\in \Delta(\C)\setminus \C$ and define $c_0 = \bigcap_{c\in\Tk_\C(\sigma)} c$. The trunk $\Tk_\C(\sigma)$ is nonempty since $\sigma\in \Delta(\C)$, and $c_0\in \C$ since $\C$ is intersection complete. \edit{By construction, $\sigma$ is a subset of $c_0$. In fact,} $\sigma$ is a proper subset of $c_0$ since $\sigma\notin\C$. Thus we may choose $i\in c_0\setminus \sigma$. 

We claim that $U_\sigma\subseteq U_i$. Indeed, since $c_0$ is the unique minimal element of $\Tk_\C(\sigma)$, every codeword containing $\sigma$ also contains $i$. This implies that $U_\sigma\subseteq U_i$.

For the converse, we prove the contrapositive. Suppose that $\C$ is not intersection complete, so there exist $c_1$ and $c_2$ in $\C$ such that $c_1\cap c_2\notin \C$. Define $\sigma = c_1\cap c_2$ and note that $\sigma\in \Delta(\C)\setminus \C$.  Then choose any (possibly non-convex) realization $\U = \{U_1,\ldots, U_n\}$ of $\C$, and let $i\in [n]\setminus \sigma$. Observe that $i$ is contained in at most one of $c_1$ and $c_2$. Since $U_\sigma$ contains $U_{c_1}$ and $U_{c_2}$, it follows that there is a point in $U_\sigma$ that is not contained in $U_i$. This proves the result.  \end{proof}

In addition to Proposition \ref{prop:singlecover}, we will need the following ``trimming" operation, which was also employed in \cite{sparse}. 

\begin{definition}\label{def:trim}
Let $U\subseteq \R^d$ be any set and $\varepsilon>0$. The \emph{trim} of $U$ by $\varepsilon$ is the set \[
\trim(U,\varepsilon) \od \{p\in U\mid B_\varepsilon(p)\subseteq U\},
\]
where $B_\varepsilon(p)$ is the closed ball of radius $\varepsilon$ centered at $p$.
\end{definition}

\begin{proposition}\label{prop:trimconvex}
If $U\subseteq \R^d$ is convex and open, then $\trim(U,\varepsilon)$ is convex and open for any $\varepsilon >0$. Moreover, $\cl(\trim(U,\varepsilon))\subseteq U$. 
\end{proposition}
\begin{proof}
Let $p\in \trim(U,\varepsilon)$. Since $B_\varepsilon(p)$ is a closed subset of $U$ and $U$ is open and convex, there exists $\delta > 0$ such that $B_{\varepsilon+\delta}(p) \subseteq U$. This implies that the open ball of radius $\delta$ centered at $p$ is contained in $\trim(U,\varepsilon)$. Thus $p$ is an interior point of $\trim(U, \varepsilon)$, so $\trim(U,\varepsilon)$ is open.

Next let $p$ and $q$ be points in $\trim(U,\varepsilon)$. By convexity of $U$, the Minkowski sum $C = \overline{pq} + B_\varepsilon(0)$ is contained in $U$. For any $r$ on $\overline{pq}$, this implies that $B_\varepsilon(r) \subseteq C\subseteq U$. Thus $r$ lies in $\trim(U,\varepsilon)$, proving that $\trim(U,\varepsilon)$ is convex.

For the final statement, observe that no boundary point of $U$ is a boundary point of $\trim(U,\varepsilon)$. Thus all boundary points of $\trim(U,\varepsilon)$ lie in $U$, and \edit{so} the closure  $\cl(\trim(U,\varepsilon))$ must be a subset of $U$. 
\end{proof}

\begin{proposition}\label{prop:trimcommutes}
Let $U$ and $V$ be sets in $\R^d$. Then $\trim(U\cap V,\varepsilon) = \trim(U,\varepsilon)\cap \trim(V,\varepsilon)$. If $U\subseteq V$, then $\trim(U,\varepsilon)\subseteq \trim(V,\varepsilon)$. 
\end{proposition}
\begin{proof}
The first statement follows from the fact that $B_\varepsilon(p)$ is contained in both $U$ and $V$ if and only if it is contained in their intersection. The second statement is immediate from Definition \ref{def:trim}.
\end{proof}

A notion of non-degeneracy for realizations was introduced in \cite{openclosed}. Intuitively, non-degeneracy requires that the different regions in the realization do not get too close to one another, unless they intersect. The formal definition is given below.

\begin{definition}[\cite{openclosed}]\label{def:nondegen}
A collection $\U = \{U_1,\ldots, U_n\}$ of convex sets in $\R^d$ is called \emph{non-degenerate} if the following two conditions hold:\begin{itemize}
\item[(i)] For all $\sigma\in \code(\U)$, the atom $\A_\U^\sigma$ is top dimensional (i.e. its intersection with any open set is either empty, or has nonempty interior).
\item[(ii)] For all nonempty $\sigma\subseteq [n]$, we have $\bigcap_{i\in\sigma} \partial U_i \subseteq \partial U_\sigma$.
\end{itemize}
\end{definition}

When $\U$ is a collection of convex open sets, \cite{openclosed} proved that (ii) implies (i). We will show that trimming a \edit{convex} open realization of an intersection complete code $\C$ by a sufficiently small $\varepsilon$ yields a non-degenerate realization of $\C$.

\begin{lemma}\label{lem:trimrealization}
Let $\C\subseteq 2^{[n]}$ be an intersection complete code, and let $\U = \{U_1,\ldots, U_n\}$ be a convex open realization of $\C$. Then there exists $\varepsilon >0$ such that the sets $V_i = \trim(U_i,\varepsilon)$ form a non-degenerate convex open realization of $\C$.\end{lemma}
\begin{proof}
For each codeword $c\in\C$, choose a point $p_c\in \A_\U^c$. Observe that we may choose $\varepsilon$ small enough that $B_\varepsilon(p_c)\subseteq U_c$ for all $c\in\C$. We claim that this suffices.
Note that by choice of $\varepsilon$, $p_c\in V_c$ for all $c\in\C$. In particular, if $U_\sigma$ is nonempty then so is $V_\sigma$.

To prove that $\C = \code(\{V_1,\ldots, V_n\})$, we must show for all nonempty $\sigma\subseteq [n]$ that $U_\sigma$ is covered by $\{U_i\mid i\in[n]\setminus \sigma\}$ if and only if $V_\sigma$ is covered by $\{V_i\mid i\in [n]\setminus \sigma\}$. Suppose first that $U_\sigma$ is covered by $\{U_i\mid i\in[n]\setminus \sigma\}$. \edit{Since $\C$ is intersection complete,} Proposition \ref{prop:singlecover} \edit{implies that} there exists some $i\in[n]\setminus \sigma$ with $U_\sigma\subseteq U_i$. But \edit{by Proposition \ref{prop:trimcommutes}} trimming commutes with intersections and preserves containment, \edit{and so} $V_\sigma\subseteq V_i$ as desired.

For the converse, we prove the contrapositive. Suppose that $U_\sigma$ is not covered by $\{U_i\mid i\in[n]\setminus \sigma\}$. Then $\sigma\in \C$ and we may consider the point $p_\sigma$. By choice of $\varepsilon$, $p_\sigma$ is in $V_\sigma$ but not any $U_i$ with $i\in[n]\setminus \sigma$. Since $V_i\subseteq U_i$, this implies that $p_\sigma$ is not covered by $\{V_i\mid i\in[n]\setminus \sigma\}$. This proves that $\C = \code(\{V_1,\ldots, V_n\})$. 

To see that the $V_i$ form a non-degenerate realization, we must check (ii) of Definition \ref{def:nondegen}. For any nonempty $\sigma\subseteq [n]$, let $p$ be a point in $\bigcap_{i\in\sigma} \partial V_i$. Observe that since the closure of any $V_i$ is contained in $U_i$, the point $p$ lies in $U_\sigma$. \edit{Hence} $U_\sigma$ is nonempty. We may choose a point $q\in V_\sigma$, and consider the line segment $\overline{pq}$. Since $p$ is a boundary point of all $V_i$ with $i\in\sigma$, the line segment $\overline{pq}$ is contained in $V_i$ except for the point $p$. But this implies that all points on the line segment except $p$ lie in $V_\sigma$. Thus $p$ is a boundary point of $V_\sigma$ and the result follows.\end{proof}

\edit{\begin{remark}
Note that we only used intersection completeness of $\C$ once in the proof of Lemma \ref{lem:trimrealization}, namely in the second paragraph so that we could apply Proposition \ref{prop:singlecover}. The third paragraph of the proof shows that, independent of intersection completeness, the codewords of $\C$ all appear as codewords in the trimmed realization. Likewise, the fourth paragraph proves non-degeneracy of the trimmed realization without using intersection completeness. Thus if $\C$ is \emph{any} code with a convex open realization $\U$, then we may trim $\U$ by some small $\varepsilon >0$  to obtain a non-degenerate realization of a code $\D$ with $\C\subseteq \D \subseteq \Delta(\C)$. 
\end{remark}}

\begin{example}\label{ex:trim}
Below we show the construction used in \edit{the proof of} Lemma \ref{lem:trimrealization} (and \edit{in turn in the proof of} Theorem \ref{thm:cdimleodim}) for two realizations of intersection complete codes. 

The first is the code $\{\mathbf{123}, 1,2,3,\emptyset\}$. In \edit{the figure below} the $U_i$ already formed a non-degenerate \edit{open} realization, but trimming them slightly does not \edit{change the realized code}.
\[
\includegraphics[width=25em]{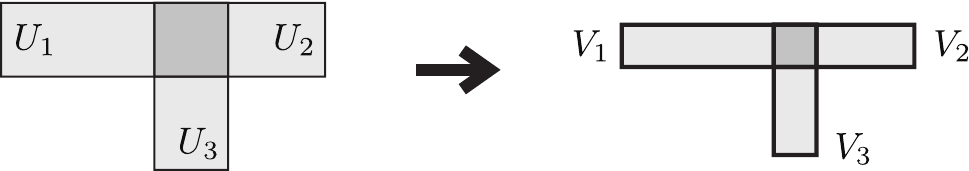}
\]
The figure below shows a degenerate \edit{open} realization of the code $\{\mathbf{13},\mathbf{14},\mathbf{23},\mathbf{24}, 1,2,3,4,\emptyset\}$. This realization is degenerate since $U_1$ and $U_2$ are disjoint but share boundary points, and similarly for $U_3$ and $U_4$. On the lefthand side, we have labeled the regions corresponding to maximal codewords. \edit{On the right we have labeled each of the sets in the trimmed realization.} 
\[
\includegraphics[width=21em]{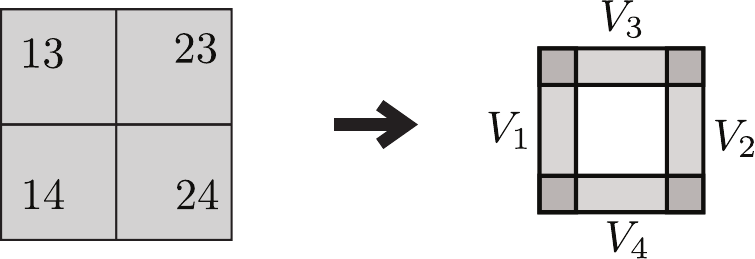}
\]
\end{example}

The importance of non-degeneracy is the following: when $\U$ is a non-degenerate collection of convex open sets, taking the closures of these sets does not change the code of the collection (see \cite[Theorem 2.12]{openclosed}). With this, we are ready to prove Theorem \ref{thm:cdimleodim}.

\begin{restatetheorem}{\ref{thm:cdimleodim}}
Let $\C\subseteq 2^{[n]}$ be an intersection complete code. Then $\cdim(\C)\le \odim(\C)$.
\end{restatetheorem}
\begin{proof}
Let $\U = \{U_1,\ldots, U_n\}$ be an open realization of $\C$ in $\R^{\odim(\C)}$. By Lemma \ref{lem:trimrealization}, we may assume that $\U$ is non-degenerate by possibly trimming the sets in the realization. By \cite[Theorem 2.12]{openclosed}, the realization consisting of closures of the $U_i$ is a closed convex realization of $\C$. Thus $\cdim(\C)\le \odim(\C)$. 
\end{proof}

\begin{example}\label{ex:trimfails}
Trimming a\edit{n open} realization may fail when a code is not intersection complete. The following shows a\edit{n open convex} realization of the code $\{\mathbf{123}, 12, 13, \emptyset\}$ with labeled atoms, and a trimming of that realization. One can observe that no matter how small we choose $\varepsilon$, trimming this realization always yields an arrangement in which part of $V_1$ is not covered by $V_2$ and $V_3$. 
\[
\includegraphics[width=30em]{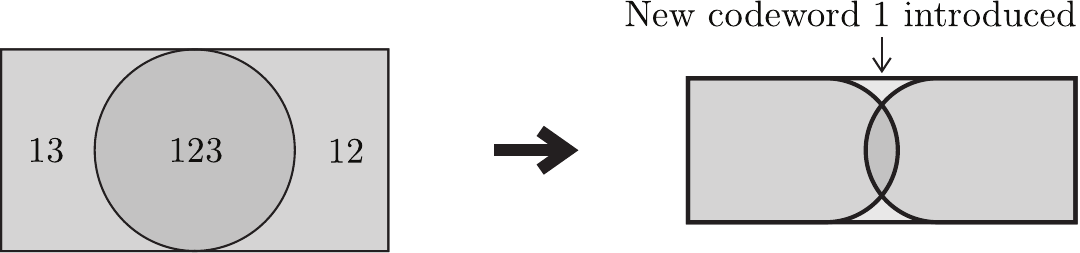}
\]
Of course, we could have drawn a better realization of this code. However, we are not always so lucky. There are examples of open convex codes where trimming will fail for \emph{any} convex open realization---see for example \cite[Section 2.3]{openclosed}, which describes a code for which every convex open realization is degenerate.
\end{example}

\section{Closed Embedding Dimension is Bounded by $\min\{2d+1, n-1\}$}\label{sec:cdimlinear}

Throughout this section, let us fix a (possibly not intersection complete) code $\C\subseteq 2^{[n]}$, and let $d = \dim(\Delta(\C))$. \edit{Let us also assume that $n\ge 2$. This avoids the possibility that $\C = \{\mathbf{1}, \emptyset\}$ (this code is intersection complete, but does not satisfy Theorem \ref{thm:cdimlinear} because its closed embedding dimension is 1, while $\min\{2d+1, n-1\} = 0$).}

We will attempt to build a realization of $\C$ using closed convex sets satisfying the bound of Theorem \ref{thm:cdimlinear}. As we will prove in Lemma \ref{lem:itworks}, this construction will succeed if and only if $\C$ is intersection complete. This result echoes \cite[Lemma 5.9]{openclosed}, but our approach allows us stronger control over the dimension of the ambient space. Our approach is inspired by the construction described in \cite[Theorem 3.1]{tancer}.

Throughout this section we will refer to the \emph{intersection completion} of $\C$, which is the code containing all intersections of codewords in $\C$. Note that $\C$ is intersection complete if and only if it is equal to its intersection completion. To begin building our attempted realization, we need to introduce several \edit{combinatorial} objects. 
\begin{lemma}\label{lem:polycomplex}
Let $m = \min \{2d+1, n-1\}$. There exists a pure, full-dimensional polytopal complex $\P$ in $\R^m$ with \edit{$n$} facets $\{P_1,\ldots, P_n\}$ \edit{and the following property: each choice of} $d+1$ \edit{distinct} facets of $\P$ \edit{determines} a unique nonempty face of $\P$ \edit{by taking the intersection of the chosen facets}. In particular, $\code(\{P_1,\ldots, P_n\})$  contains all $\sigma\subseteq [n]$ with $|\sigma| \le d+1$.
\end{lemma}
\begin{proof}
First, recall that there exists a \edit{full-dimensional} $(d+1)$-neighborly polytope in $\R^{m+1}$ with $n+1$ vertices. When $m=2d+1$, one example is the cyclic polytope, and when $m=n-1$ the  $n$-simplex suffices. Let $P\subseteq \R^{m+1}$ be a polytope dual to a $(d+1)$-neighborly polytope with $n+1$ vertices. Let $F_1,\ldots, F_n, F_{n+1}$ be the facets of $P$, and observe \edit{by neighborliness of the dual} that any $d+1$ facets of $P$ meet in a unique \edit{nonempty} face of $P$ \edit{(nonemptiness follows from the fact that in the neighborly polytope dual to $P$, any $d+1$ or fewer vertices are the vertices of a face which is not the whole polytope)}. Consider the Schlegel diagram of $P$ in $\R^m$ based at the facet $F_{n+1}$. For $1\le i\le n$, define $P_i$ to be the image of $F_i$ in the Schlegel diagram. We claim that the complex $\P$ with facets $\{P_1,\ldots, P_n\}$ is the desired polytopal complex.

Each $P_i$ is full-dimensional since each $F_i$ has dimension $m$. \edit{Hence $\P$ is pure and full-dimensional.} Furthermore, \edit{for any nonempty} $\sigma\subseteq [n]$ \edit{with} $|\sigma| \le d+1$, the facets $P_i$ with $i\in\sigma$ meet at a unique \edit{nonempty} face \edit{of $\P$ (since the Schlegel diagram preserves intersections and nonemptiness).} A point in the relative interior of this face will not lie in any $P_j$ with $j\notin\sigma$, and so $\sigma\in \code(\P)$. This proves the result.
\end{proof}
For the remainder of this section, let us fix a polytopal complex $\P$ with facets $\{P_1,\ldots, P_n\}$ as given by Lemma \ref{lem:polycomplex}. So far we have a fixed code $\C$, and a fixed complex $\P$. We begin to relate these two objects to \edit{each other} below.  \edit{For $\sigma\subseteq[n]$, recall that $P_\sigma$ denotes $\bigcap_{i\in\sigma} P_i$.}
\begin{itemize}
\item \edit{For each nonempty $\sigma\subseteq [n]$ with $|\sigma|\le d+1$, fix a point $p_\sigma$ in the relative interior of $P_\sigma$, and}
\item \edit{For each $i\in[n]$, define $V_i \od \conv\{p_c\mid c\in \Tk_\C(i)\}$. }
\end{itemize}

These objects are illustrated in Example \ref{ex:2d+1} below. \edit{For now, we observe two facts that we will make use of in the lemmas below:}\begin{enumerate}
\item Since the various $P_\sigma$ with $|\sigma|\le d+1$ are distinct faces of $\P$, $p_\sigma \in P_i$ if and only if $i\in\sigma$, \edit{and}
\item \edit{For all $i\in[n]$}, $V_i\subseteq P_i$, and as a consequence $V_{\tau}\subseteq P_{\tau}$ for all nonempty $\tau\subseteq [n]$.
\end{enumerate}
The following lemmas build the connection between the sets $V_i$ and the structure of our \edit{fixed} code $\C$.

\begin{lemma}\label{lem:Hsupport}
Let $\sigma\subseteq [n]$ with $|\sigma|\ge 2$, and let $i\in\sigma$. Let $H$ be a supporting hyperplane for the proper face $P_\sigma$ of $P_i$. Then $V_i\cap H  = \conv\{p_c\mid c\in\Tk_\C(\sigma)\}$.
\end{lemma}
\begin{proof}
Consider the points $\{p_c\mid c\in \Tk_\C(i)\}$, the convex hull of which is equal to $V_i$  by definition. Since $V_i\subseteq P_i$, we see that $V_i\subseteq H^\ge$. Thus $V_i\cap H$ is the convex hull of all points in $\{p_c\mid c\in \Tk_\C(i)\}$ which lie in $H$. If $c\in\Tk_\C(i)$ but $\sigma\not\subseteq c$, then we may choose $j\in \sigma\setminus c$, noting that $p_c\notin P_j$. In particular, $p_c\in P_i$ but $p_c\notin P_\sigma$. Thus $p_c$ lies in $H^>$ when $\sigma\not\subseteq c$. On the other hand, if $\sigma\subseteq c$ then $p_c\in P_\sigma\subseteq H$. Thus $V_i\cap H$ is the convex hull of  $\{p_c\mid c\in \Tk_\C(\sigma)\}$ as desired.
\end{proof}

\begin{lemma}\label{lem:trunkconv}
Let $\sigma\subseteq [n]$ be nonempty. Then $V_\sigma = \conv\{p_c\mid c\in\Tk_\C(\sigma)\}$.
\end{lemma}
\begin{proof}
Let $C = \conv\{p_c\mid c\in\Tk_\C(\sigma)\}$. Then $C\subseteq V_\sigma$ since each $p_c$ with $c\in\Tk_\C(\sigma)$ lies in $V_j$ for all $j\in\sigma$. For the reverse inclusion, we consider two cases. If $\sigma= \{i\}$ then $C = V_i$ and the result is immediate. Otherwise, $|\sigma|\ge 2$ and we may choose $i\in\sigma$ and $H$ a supporting hyperplane for the face $P_\sigma$ of $P_i$. Observe that $V_\sigma\subseteq V_i\cap P_\sigma \subseteq V_i\cap H$, and by Lemma \ref{lem:Hsupport} $V_i\cap H = C$, proving the result.
\end{proof}

\begin{lemma}\label{lem:faces}
Let $\sigma$ and $\tau$ be nonempty subsets of $[n]$. Then $V_\sigma$ is a face of $V_\tau$ if and only if $\Tk_\C(\sigma)\subseteq \Tk_\C(\tau)$. \edit{Furthermore, $V_\sigma$ is a proper face of $V_\tau$ if and only if $\Tk_\C(\sigma)$ is a proper subset of $\Tk_\C(\tau)$.}
\end{lemma}
\begin{proof}
First suppose that $\Tk_\C(\sigma)\subseteq \Tk_\C(\tau)$. This implies that every codeword that contains $\sigma$ also contains $\tau$, and so $\Tk_\C(\sigma) = \Tk_\C(\sigma\cup \tau)$. Lemma \ref{lem:trunkconv} then implies that $V_\sigma = V_{\sigma\cup\tau}$, and so it suffices to prove that $V_{\sigma\cup\tau}$ is a face of $V_\tau$. We may reduce to the case in which $\tau\subseteq \sigma$, \edit{that is,} it suffice\edit{s} to prove that $V_\sigma$ is a face of all $V_i$ with $i\in\tau$. If $\sigma = \{i\}$ then $\tau= \{i\}$ and the result is immediate. Otherwise, $|\sigma|\ge 2$, and for any $i\in\tau$ we may choose a hyperplane $H$ supporting the face $P_\sigma$ of $P_i$. Lemma \ref{lem:Hsupport} implies that $H\cap V_i = \conv\{p_c\mid c\in\Tk_\C(\sigma)\}$, and Lemma \ref{lem:trunkconv} implies that this \edit{set} is $V_\sigma$. Thus $V_i\cap H = V_\sigma$ and $V_\sigma$ is a face of $V_i$ for all $i\in\tau$ as desired.

\edit{To prove that $\Tk_\C(\sigma)\subseteq \Tk_\C(\tau)$ whenever $V_\sigma$ is a face of $V_\tau$}, we argue by contrapositive. If $\Tk_\C(\sigma)\not\subseteq \Tk_\C(\tau)$ then there exists $c\in\C$ with $\sigma\subseteq c$ but $\tau\not\subseteq c$. Consider the point $p_c$. Since $\tau\not\subseteq c$, there exists $i\in \tau\setminus c$, and we see that $p_c\notin P_i$. But $V_\tau\subseteq V_i\subseteq P_i$, so $p_c\notin V_\tau$. On the other hand, $p_c\in V_\sigma$, so $V_\sigma$ is not contained in $V_\tau$, \edit{and thus cannot be a face of $V_\tau$}.

\edit{We can now prove the final sentence in the lemma. If $V_\sigma$ is a proper face of $V_\tau$, then $\Tk_\C(\sigma)$ is a subset of $\Tk_\C(\tau)$. If this containment were not proper, then $V_\tau$ would be a face of $V_\sigma$, a contradiction. Conversely, if $\Tk_\C(\sigma)$ is a proper subset of $\Tk_\C(\tau)$ then $V_\sigma$ is a face of $V_\tau$. If $V_\sigma$ were not a proper face, then we would have $V_\sigma = V_\tau$, which in turn implies $\Tk_\C(\tau)\subseteq \Tk_\C(\sigma)$, a contradiction. This proves the result. }
\end{proof}
\begin{lemma}\label{lem:properstuff}
Let $\sigma\subseteq [n]$ be nonempty. Then $\sigma$ lies in the intersection completion of $\C$ if and only if the following holds: $\Tk_\C(\sigma)$ is nonempty and properly contains $\Tk_\C(\sigma\cup\{i\})$ for all $i\in[n]\setminus \sigma$. 
\end{lemma}
\begin{proof}
If $\sigma$ is an intersection of codewords in $\C$, then there must be a codeword containing $\sigma$, and thus $\Tk_\C(\sigma)$ is nonempty. If there exists $i\in[n]\setminus \sigma$ such that $\Tk_\C(\sigma) = \Tk_\C(\sigma\cup\{i\})$, then every codeword of $\C$ containing $\sigma$ also contains $i$. This is a contradiction, since $\sigma$ is the intersection of all codewords in $\C$ that contain it.

For the converse we consider two cases. If $\sigma = [n]$ and $\Tk_\C(\sigma)$ is nonempty then $[n]\in \C$ and the result follows. Otherwise $\sigma$ is a proper subset of $[n]$. Since $\Tk_\C(\sigma)$ is nonempty and properly contains $\Tk_\C(\sigma\cup\{i\})$ for all $i\in[n]\setminus \sigma$, for every $i\in[n]\setminus \sigma$ we may choose a codeword $c_i$ with $\sigma\subseteq c_i$ and $i\notin c_i$. The intersection of all such $c_i$ is $\sigma$, proving the result. 
\end{proof}

\begin{lemma}\label{lem:itworks}The set $\V = \{V_1,\ldots, V_n\}$ is a closed realization of the intersection completion of $\C$. In particular, $\V$ is a realization of $\C$ if and only if $\C$ is intersection complete. 
\end{lemma}
\begin{proof}
Let $\widehat{\C}$ denote the intersection completion of $\C$. We argue for each nonempty $\sigma\subseteq[n]$ that $\sigma\in \widehat\C$ if and only if $\sigma\in \code(\V)$. By Lemma \ref{lem:properstuff} it suffices to argue that $\sigma\in\code(\V)$ if and only if $\Tk_\C(\sigma)$ is nonempty and $\Tk_\C(\sigma\cup\{i\})$ is a proper subset of $\Tk_\C(\sigma)$ for all $i\in[n]\setminus\sigma$. By \edit{Lemma \ref{lem:trunkconv} and} Lemma \ref{lem:faces}, this condition is equivalent to the requirement that $V_\sigma$ is nonempty, and $V_{\sigma\cup\{i\}}$ is a proper face of $V_\sigma$ for all $i\in[n]\setminus \sigma$. This is in turn equivalent to the statement that $V_\sigma$ is nonempty and not covered by $\{V_i\mid i\in[n]\setminus\sigma\}$, which happens if and only if $\sigma\in \code$\edit{(}$\V$\edit{)}, proving the result.
\end{proof}

\begin{example}\label{ex:2d+1} To make the construction in Lemma \ref{lem:itworks} concrete, we give an example for the intersection complete code $\C = \{\mathbf{123}, 12, 1,2,3,\emptyset\}$. We choose $\P$ in $\R^2$ with facets $P_1, P_2, P_3$ which are triangles meeting at a common vertex. This is shown below \edit{on the left}, and the various $p_\sigma$ are \edit{shown with labeled} dots. \edit{The righthand side of the figure illustrates the closed convex realization $\V = \{V_1, V_2, V_3\}$ constructed in the proof of Lemma \ref{lem:itworks}.} The sets $V_1$ and $V_2$ are triangles, and $V_3$ is the line segment from $p_{3}$ to $p_{123}$.\[
\includegraphics{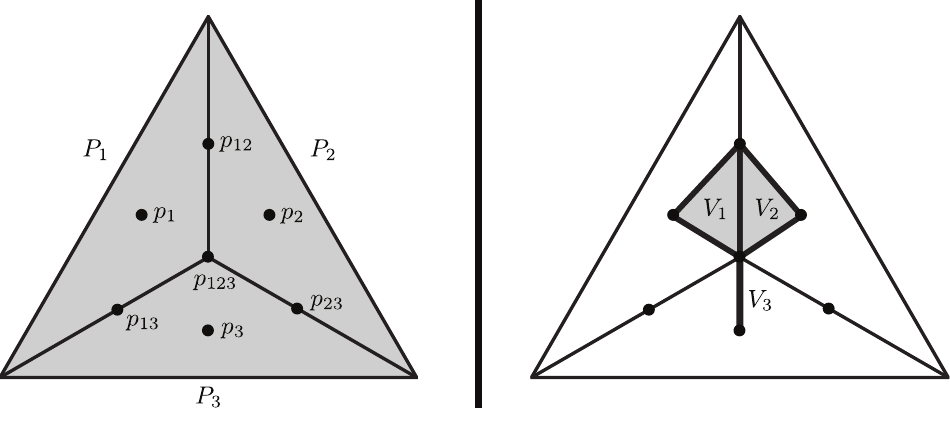}
\]
\end{example}

\begin{restatetheorem}{\ref{thm:cdimlinear}}
Let $\C\subseteq 2^{[n]}$ be an intersection complete code \edit{with $n\ge 2$}, and \edit{let} $d=\dim(\Delta(\C))$. Then $\cdim(\C)\le \min \{2d+1, n-1\}$.
\end{restatetheorem}
\begin{proof}
In this section we have chosen a polytopal complex $\P$ in $\R^{\min\{2d+1,n-1\}}$, and used it to construct a collection $\V = \{V_1,\ldots, V_n\}$ of closed convex sets. Lemma \ref{lem:itworks} says that $\V$ realizes $\C$ if and only if $\C$ is intersection complete. This proves the result. 
\end{proof}
In Section \ref{sec:SDelta}, we will see that this bound on closed embedding dimension may fail dramatically for open embedding dimension. Before proving this, we use Section \ref{sec:sunflowercodeversion} to recall a theorem from \cite{sunflowers}, and show that it is equivalent to a statement about the open embedding dimension of a family of intersection complete codes.  
\section{A Code Version of the Sunflower Theorem}\label{sec:sunflowercodeversion}

In this section we recall a result regarding sunflowers of convex open sets. In Section \ref{sec:SDelta}, we will use this result to build a family of intersection complete codes with large open embedding dimension.

\begin{definition}\label{def:sunflower}
Let $\U = \{U_1,\ldots, U_n\}$ be a collection  of convex sets in $\R^d$ and let $\C = \code(\U)$. The collection $\U$ is called a \emph{sunflower} if $[n]\in\C$, and $\C\setminus \{[n]\}$ contains \edit{only} codewords of weight at most 1. That is, a sunflower is just a $1$-flexible sunflower. As in Definition \ref{def:flexiblesunflower}, we will call the $U_i$ \emph{petals} and $U_{[n]}$ will be called the \emph{center} of $\U$. \end{definition}

\begin{theorem}[Sunflower Theorem, \cite{sunflowers}]\label{thm:sunflower}
Let $d\ge 1$, let $\U = \{U_1,\ldots, U_{d+1}\}$ be a convex open sunflower in $\R^d$, and for each $i\in[d+1]$ choose a point $p_i\in U_i$. Then $\conv\{p_1,\ldots, p_{d+1}\}$ contains a point in the center of $\U$. 
\end{theorem}

Note that the result above fails when we consider a sunflower with $d$ petals in $\R^d$. In particular, one may take an infinite rectangular \edit{prism} about each coordinate axis to form a sunflower whose center is a hypercube at the origin. In this situation, choosing the $p_i$ \edit{to be} sufficiently large positive \edit{multiples of the coordinate basis vectors} yields points in each petal whose convex hull does not touch the center of the sunflower. 

The sunflower theorem may be restated purely in the language of convex codes. We do this below in order to simplify our discussion in the following section, and also to foreshadow our applications of Theorem \ref{thm:flexible} in Section \ref{subsec:SCoverD}.

\begin{definition}\label{def:Sn}
For $n\ge 1$, define  $\S_n\subseteq 2^{[n+1]}$ to be the code consisting of the following codewords: $[n]$, all singleton sets, all pairs $\{i, n+1\}$ for $1\le i \le n$, and the empty set. 
\end{definition} 

\begin{example}
Let us look at the first few $\S_n$:\begin{align*}
\S_1 = & \{\mathbf{12}, 1, 2, \emptyset\},\\
\S_2 = &\{\mathbf{12},\mathbf{13},\mathbf{23}, 1, 2, 3, \emptyset\},\\
\S_3 = &\{\mathbf{123}, \mathbf{14},\mathbf{24},\mathbf{34}, 1, 2, 3, 4, \emptyset\}.
\end{align*} These have \edit{convex open} realizations in $\R^1,\R^2$, and $\R^3$ respectively, illustrated below. Theorem \ref{thm:sunflowercodeversion} \edit{below} says that these realizations are minimal in dimension. 
\[
\includegraphics[width=31em]{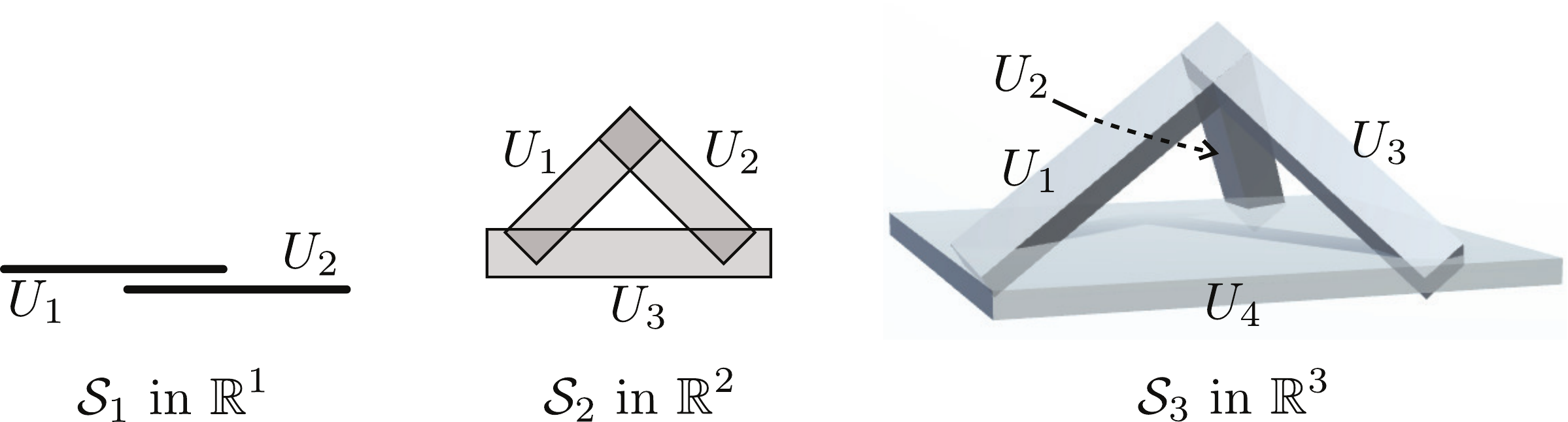}
\]
\end{example}

Note that $\S_n$ is an intersection complete code. The sunflower theorem can be restated as follows:
\begin{theorem}[Sunflower Theorem, Code Version]\label{thm:sunflowercodeversion}
For all $n\ge 1$, $\odim(\S_n) = n$. 
\end{theorem}
\begin{proof}
When $n=1$, we have $\S_n = \{12, 1, 2, \emptyset\}$, which can be realized by two overlapping intervals in $\R^1$. For $n\ge 2$, $\S_n$ has $n+1$ maximal codewords, and so by  \cite[Theorem 1.2]{openclosed} $\S_n$ has an open realization in $\R^n$. We will show that it does not have an open realization in $\R^{n-1}$. Suppose for contradiction that there exists an open realization  $\U = \{U_1,\ldots, U_{n+1}\}$ of $\S_n$ in $\R^{n-1}$. Observe that $\{U_1,\ldots, U_n\}$ is a sunflower, and $U_{n+1}$ intersects $U_i$ for all $i\in[n]$. Thus for each $i\in[n]$ we may choose $p_i\in U_i\cap U_{n+1}$. The convex hull $\conv\{p_1,\ldots, p_n\}$ is contained in $U_{n+1}$, but by Theorem \ref{thm:sunflower} this convex hull also meets $U_{[n]}$. Thus $U_{n+1}\cap U_{[n]}$ is nonempty. Since $[n+1]$ is not a codeword in $\S_n$ this is a contradiction. 
\end{proof}

In the following section we build on the family $\S_n$ to construct a family of intersection complete codes on $n$ neurons whose open embedding dimension is $\binom{n-1}{\lfloor (n-1)/2\rfloor}$. 
\section{A Family of Codes with Large Open Embedding Dimension}\label{sec:SDelta}

In this section, we will associate to every simplicial complex $\Delta\subseteq 2^{[n]}$ an intersection complete code $\S_\Delta\subseteq 2^{[n+1]}$. As long as $\Delta$ has at least two facets, the open embedding dimension of $\S_\Delta$ is exactly the number of facets in $\Delta$. \edit{To start, recall the following definition.}

\begin{restatedefinition}{\ref{def:SDelta}}
Let $\Delta\subseteq 2^{[n]}$ be a simplicial complex. Define $\S_\Delta\subseteq 2^{[n+1]}$ to be the code \[
\S_\Delta \od \left(\Delta \ast (n+1)\right)\cup\{[n]\},
\]
where $\Delta\ast (n+1)$ denotes the cone over $\Delta$ with apex $n+1$. 
\end{restatedefinition}

We start with some straightforward structural observations about the code $\S_\Delta$.

\begin{proposition}\label{prop:SDelta}
\edit{The code} $\S_\Delta$ is intersection complete. If $\Delta\subsetneq 2^{[n]}$ and has $m$ facets, then $\S_\Delta$ has $m+1$ maximal codewords. In particular, $\odim(\S_\Delta)\le \max\{2, m\}$.
\end{proposition}
\begin{proof}
First note that $\S_\Delta$ is a simplicial complex, plus the codeword $[n]$. Adding a single codeword to a simplicial complex always yields an intersection complete code, so $\S_\Delta$ is intersection complete. 

Let $F_1,\ldots, F_{m}$ be the facets of $\Delta$. Observe that \edit{each} maximal codeword of $\S_\Delta$ \edit{is} either \edit{a} facet of $\Delta\ast (n+1)$, or equal to $[n]$. The facets of $\Delta\ast(n+1)$ are just $F_i\cup \{n+1\}$ for $i\in[m]$. Since $\Delta\subsetneq 2^{[n]}$, $[n]$ is \edit{an additional} a maximal codeword of $\S_\Delta$, so $\S_\Delta$ has $m+1$ maximal codewords in total. The bound $\odim(\S_\Delta)\le \max\{2, m\}$ then follows immediately from \cite[Theorem 1.2]{openclosed}.
\end{proof}

\begin{restatetheorem}{\ref{thm:SDelta}}
Let $\Delta\subseteq 2^{[n]}$ be a simplicial complex with $m\ge 2$ facets. Then $\S_\Delta$ (as given by Definition \ref{def:SDelta}) is an intersection complete code with $m+1$ maximal codewords, and $\odim(\S_\Delta) = m$.
\end{restatetheorem}

\begin{proof} By Proposition \ref{prop:SDelta} we know that $\S_\Delta$ is intersection complete, has $m+1$ maximal codewords, and \edit{satisfies} $\odim(\S_\Delta) \le m$. Thus it suffices to show that $\S_\Delta$ does not have an open realization in $\R^{m-1}$. Suppose for contradiction that we had such a realization $\U = \{U_1,\ldots, U_{n+1}\}$.

Label the facets of $\Delta$ as $F_1,\ldots, F_{m}$, and for each $i\in[m]$ define $V_i = U_{F_i}$. Lastly, define $V_{m+1} = U_{n+1}$. Now observe that the pairwise intersection of any two distinct $V_i$ with $i\in[m]$ is $U_{[n]}$ \edit{since $[n]$ is the only codeword of $\S_\Delta$ that properly contains more than one $F_i$. Thus} $\{V_1,\ldots, V_{m}\}$ is a sunflower. Note that $V_{m+1}$ intersects each petal of this sunflower since $F_i\cup\{n+1\}$ is a codeword of $\S_\Delta$ for all $i\in[m]$. However, $V_{m+1}$ does not intersect $U_{[n]}$ \edit{since $[n]$ is not a face of $\Delta$ (recall $m\ge 2$). 

We have constructed an open sunflower $\{V_1,\ldots, V_m\}$ in $\R^{m-1}$ and a convex open set $V_{m+1}$ which intersects each petal of the sunflower but not its center. Equivalently,}  $\{V_1,\ldots, V_{m+1}\}$ is a convex open realization of $\S_{m}$ in $\R^{m-1}$. This contradicts Theorem \ref{thm:sunflowercodeversion}, and so $\odim(\S_\Delta)$ must be equal to $m$ as desired. 
\end{proof}

\begin{corollary}\label{cor:choice}
For any $n\ge 2$ and $1\le m\le \binom{n-1}{\lfloor (n-1)/2\rfloor}$, there exists an intersection complete code on $n$ neurons with $m+1$ maximal codewords, and open embedding dimension equal to  $m$.
\end{corollary}
\begin{proof}
For $m = 1$, the code $\{\mathbf{1}, \emptyset\}$ suffices. For $m \ge 2$ we apply Theorem \ref{thm:SDelta}. Among all $\binom{n-1}{\lfloor (n-1)/2\rfloor}$ subsets of $[n-1]$ with size $\lfloor (n-1)/2 \rfloor$, we may  select $m$. Letting $\Delta$ be the simplicial complex with these subsets as its facets, we see that $\S_\Delta$ is the desired code. 
\end{proof}

\begin{corollary}\label{cor:exponential}
There is a family of codes $\E_n\subseteq 2^{[n]}$ such that $\odim(\E_n)$ grows \edit{at least as fast as $\frac{2^{n-1}}{n}$.}
\end{corollary} 
\begin{proof}
By Corollary \ref{cor:choice}, we may choose $\E_n$ so that $\odim(\E_n) = \binom{n-1}{\lfloor (n-1)/2\rfloor}$. \edit{Observe that $n \binom{n-1}{\lfloor (n-1)/2\rfloor} \ge 2^{n-1}$ since there are no more than $\binom{n-1}{\lfloor (n-1)/2\rfloor}$ subsets of $[n-1]$ of size $i$ for $i=0,1,\ldots, n-1$. Thus } $\binom{n-1}{\lfloor (n-1)/2\rfloor} \ge \frac{2^{n-1}}{n}$, \edit{proving the result.}\end{proof}

Qualitatively, these results are very surprising. The codes $\S_\Delta$ are ``almost" simplicial complexes (we have added the single codeword $[n]$ to a simplicial complex), but their open embedding dimensions grow exponentially faster than that of any simplicial complex. Strikingly, these codes provide the first example of codes whose embedding dimension (open or closed) is larger than $n-1$. 

\begin{remark}
From the perspective of the neuroscience which motivates the study of convex codes, Corollary \ref{cor:exponential} has the following interpretation: theoretically, $n$ neurons may ``recognize" dimensions that are exponentially large in $n$. Whether such a phenomenon ever occurs in experimental data could be an interesting avenue of investigation. 
\end{remark}

\section{Flexible Sunflowers}\label{sec:flexible}

In this section our goal is to investigate $k$-flexible sunflowers of convex open sets. These are a generalization of sunflowers in which we allow petals to overlap outside the center of the sunflower, but no more than $k$ at a time. For sunflowers, we saw in Theorem \ref{thm:sunflower} that sampling a point in each petal and taking the convex hull always yielded a point in the center of the sunflower if we had enough petals relative to our ambient dimension. We will see that the same holds for $k$-flexible sunflowers, and the minimum number of petals needed is proportional to $k$, as well as the ambient dimension \edit{(see Theorem \ref{thm:flexible} below)}. Qualitatively, the more flexibility we allow in a sunflower, the larger the number of petals we need to sample in order to guarantee that the convex hull of the sampled points intersects the center of the sunflower. 

\edit{Our proofs rely heavily on the assumption that we are working with open sets, and the results of this section do not apply to sunflowers of closed convex sets. Indeed, taking $n$ line segments in the positive quadrant of $\R^2$ which meet at the origin yields a closed sunflower for which the conclusion of Theorem \ref{thm:flexible} does not hold.}

To begin, let us recall the definition of a $k$-flexible sunflower. 

\begin{restatedefinition}{\ref{def:flexiblesunflower}}
Let $\U = \{U_1,\ldots, U_n\}$ be a collection of convex sets in $\R^d$ and let $\C = \code(\U)$. The collection $\U$ is called a \emph{$k$-flexible sunflower} if $[n]\in\C$, and all other codewords have weight at most $k$. The $U_i$ are called \emph{petals} and $U_{[n]}$ is called the \emph{center} of $\U$. \end{restatedefinition}

We start with a family of examples. For each $d\ge 2$ and $k \ge 1$, Proposition \ref{prop:fails} describes a $k$-flexible sunflower $\U$ in $\R^d$ with $dk$ petals in which we can sample points from each petal whose convex hull does not contain a point in the center of $\U$.

\begin{proposition}\label{prop:fails}
For all $d\ge 2$ and $k\ge1$, there exists a\edit{n open} $k$-flexible sunflower $\U =\{U_1,\ldots, U_n\}$ in $\R^d$ with $n=dk$, and points $p_1,\ldots, p_n$ with $p_i\in U_i$, such that $\conv\{p_1,\ldots, p_n\}$ does not contain a point in the center of $\U$.
\end{proposition}
\begin{proof}
For $k=1$, we begin with an open unit hypercube in $\R^d$ centered at the origin, and let $U_i$ be the Minkowski sum of this hypercube with a line segment from the origin to a large positive multiple of $e_i$. We can see that the $U_i$ form a $d$-petal sunflower, and our desired $p_i$ are just the large multiples of $e_i$. 

For $k\ge 2$, we can take the sunflower described above and duplicate each of the $d$ petals $k$ times. This creates a $k$-flexible sunflower, and the same sampling of points (with each duplicated $k$ times) satisfies the proposition.
\end{proof}

\begin{remark}
One might argue that the construction above is unsatisfying. Should we not stipulate that petals diverge in different directions, or at least are distinct? It turns out we can address these concerns. Start with the usual coordinate-direction sunflower whose center is a unit hypercube, as described above. If $k=1$ we are done. Otherwise, choose a cyclic permutation $\sigma$ of $[d]$, for example $i\mapsto i+1$ mod $d$. Then, we can duplicate each petal in our coordinate-direction sunflower $k$ times, but when duplicating the $i$-th petal we ``skew" it slightly in the direction of $-e_{\sigma(i)}$. If each duplicated petal is skewed a different amount, our petals will diverge from one another. As long as we skew a small enough amount, this yields a $k$-flexible sunflower from which we can sample the desired $p_i$.

This construction is illustrated below for $k=3$ and $d=2$:\[
\includegraphics[width=11em]{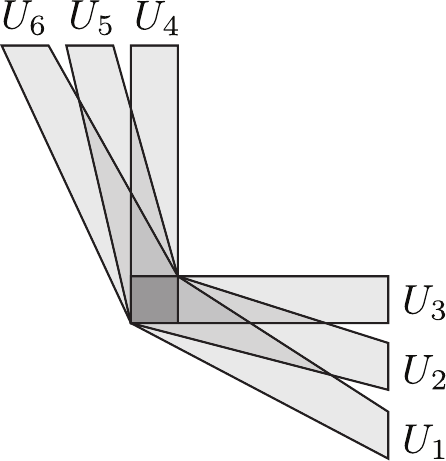}
\]
\end{remark}

We now turn our attention to proving Theorem \ref{thm:flexible}. We will see that some technical lemmas regarding $k$-flexible sunflowers together with Tverberg's theorem \edit{are enough to prove the theorem.} We start by showing that the center of every open $k$-flexible sunflower admits a set of supporting halfspaces \edit{with two properties: first, the common intersection of the halfspaces closely approximates the center (in the sense that the common intersection is contained in the closure of the center)}, and second, any one of \edit{these halfspaces} contains all but at most $k$ of the petals.

\begin{definition}\label{def:wellsupported}
Let $\U=\{U_1,\ldots, U_n\}$ be a $k$-flexible sunflower in $\R^d$ with center $U$. A point $b\in\partial U$ is called \emph{well-supported} if it is not in the boundary of $U_i\cap \partial U$ (considered as a subset of the topological space $\partial U$) for any $i\in[n]$.  
\end{definition}

\begin{lemma}\label{lem:densesupport}
Let $\U=\{U_1,\ldots, U_n\}$ be a $k$-flexible sunflower in $\R^d$ with center $U$. The set of well-supported points is dense in $\partial U$. 
\end{lemma}
\begin{proof}
Consider the sets $U_i\cap \partial U$ in $\partial U$. For each of these sets, the set of non-boundary points in $\partial U$ is dense and open when considered as a subset of $\partial U$. The set of well-supported points is just the intersection of non-boundary points of $U_i\cap \partial U$ in $\partial U$for all $i$, and a finite intersection of dense open sets is again open and dense. Thus the well-supported points are dense in $\partial U$. 
\end{proof}

\begin{lemma}\label{lem:densecut}
Let $U\subseteq \R^d$ be a convex open set. Let $B$ be a dense subset of the boundary of $U$, and for each $b\in B$ let $H_b$ be a supporting hyperplane to $U$ at $b$. Then $\bigcap_{b\in B} H_b^>$ is contained in $\overline U$. 
\end{lemma}
\begin{proof}
Consider any point $p\notin \overline U$. Since \edit{$U$ is open and} $p$ lies a positive distance away from $U$, the intersection of $\interior(\conv(\{p\} \cup U))$ with $\partial U$ is a relatively open \edit{nonempty} subset of $\partial U$, and thus contains some $b\in B$ \edit{because $B$ is dense}. Since $\interior(\conv(\{p\}\cup U))$ is open, the line segment $\overline{pb}$ can be extended so that it ends at a point $q\in U$, as shown in the following figure.\[
\includegraphics[width=18em]{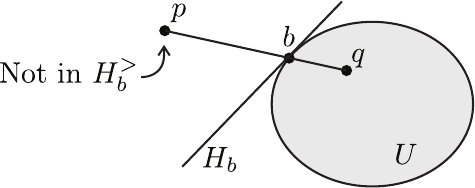}
\]Now, consider the supporting hyperplane $H_b$. We have $U\subseteq H_b^>$. In particular, $H_b^>$ contains $q$ but not $b$. Since $b$ lies between $q$ and $p$, we see that $H_b^>$ does not contain $p$. Thus $p\notin \bigcap_{b\in B} H_b^>$ and the lemma follows.  
\end{proof}

\begin{lemma}\label{lem:goodhalfspaces}
Let $\U=\{U_1,\ldots, U_n\}$ be a\edit{n open} $k$-flexible sunflower in $\R^d$ with center $U$, and let $b\in \partial U$ be well-supported. Let $H_b$ be a supporting \edit{hyperplane} for $U$ at $b$. Then $U_i\subseteq H_b^>$ for all $i\in [n]$ such that $b\notin U_i$. \edit{In particular, $H_b^>$ contains all but at most $k$ petals of $\U$.}
\end{lemma}
\begin{proof}
Suppose not, so that there exists \edit{$i\in[n]$ such that $b\notin U_i$ and} $U_i$ is not contained in $H_b^>$. Since $U_i$ is open, we may assume that there exists a point $p\in U_i$ strictly on the negative side of $H_b$. Then choose any point $q\in U$, and consider the line segment $\overline{qb}$. All points on this line segment other than $b$ lie in $U$. For each $r\in \overline{qb}$ with $r\neq b$, note that the line segment $\overline{pr}$ is contained in $U_i$. \edit{Moreover, since $U$ is convex and $q$ lies in $U$ while $p$ lies outside of $U$, this line segment intersects the boundary of $U$ in a unique point.} The set of these intersection points forms a subset of $U_i\cap \partial U$ whose closure contains $b$. This is illustrated in the figure below, with the points in $U_i\cap \partial U$ converging to $b$ shown in the bold curved line segment.
\[
\includegraphics[width=12em]{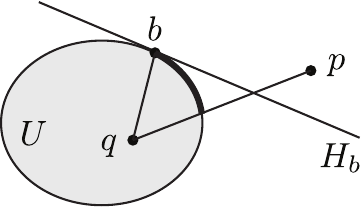}
\]
But since $b\notin U_i$, this implies that $b$ is a boundary point of $U_i\cap \partial U$ in $\partial U$. This contradicts the fact that $b$ is well-supported, \edit{ and so $U_i\subseteq H_b^>$ whenever $b\notin U_i$. Since $\U$ is $k$-flexible, $b$ belongs to no more than $k$ petals of $\U$, and thus $H_b^>$ contains all but at most $k$ petals.} 
\end{proof}

Finally, we recall Tverberg's theorem. After stating this theorem, we are ready to prove Theorem \ref{thm:flexible}.

\begin{theorem}[Tverberg's theorem] \label{thm:tverberg}
Let $d\ge1$, $r\ge 2$, and $n = (d+1)(r-1)+1$. For any  set of \edit{$n$} points $P = \{p_1,\ldots, p_n\}$ in $\R^d$, there is a partition of $P$ into $r$ parts $P_1,\ldots, P_r$ such that $\bigcap_{i=1}^r \conv(P_i) \neq \emptyset$. 
\end{theorem}

\begin{restatetheorem}{\ref{thm:flexible}}
Let $\U = \{U_1,\ldots, U_n\}$ be an open $k$-flexible sunflower in $\R^d$. Suppose that $n\ge dk+1$, and for each $i\in[n]$ let $p_i\in U_i$. Then $\conv\{p_1,\ldots, p_n\}$ contains a point in the center of $\U$. Moreover, if $d\ge 2$ this result may fail when $n < dk+1$. 
\end{restatetheorem}
\begin{proof}
It suffices to prove the first statement for $n=dk+1$. Let $U$ denote the center of $\U$. Suppose for contradiction that the theorem does not hold, so that $\conv\{p_1,\ldots, p_n\}$ does not contain a point in $U$. Since the $U_i$ are open, we may \edit{uniformly translate the $p_i$ away from $U$ by a small positive distance}, and choose a separating hyperplane  $H$ between $\conv\{p_1,\ldots, p_n\}$ and $U$ such that $H$ does not contain any boundary point of $U$. Moreover, we can replace each $p_i$ by the intersection of the line segment $\overline{p_ip}$ with $H$, so that all  $p_i$ lie inside $H$.

Now, $H$ has dimension $d-1$, so we may apply Tverberg's theorem to our points $p_i$ with $r=k+1$. We obtain a partition $P_1,\ldots, P_{k+1}$ such that $\bigcap_{i=1}^{k+1} \conv(P_i) \neq \emptyset$. Choose any point $p$ lying in this intersection, and observe that $p\in H$. 

Let $B$ be the set of well-supported points in $\partial U$, and choose supporting halfspaces $\{H_b^>\mid b\in B\}$ as per Lemma \ref{lem:goodhalfspaces}. \edit{Recall} by Lemma \ref{lem:goodhalfspaces}, \edit{that for any $b\in B$, the halfspace} $H_b^>$ contains all \edit{$U_j$ (and thus $p_j$)} except for at most $k$. In particular, there must be some $P_i$ such that $H_b^>$ contains all points in $P_i$, and hence also their convex hull. Thus $p\in H_b^>$ for all $b\in B$. But by \edit{Lemma \ref{lem:densesupport} and} Lemma \ref{lem:densecut}, this implies that $p\in \overline U$. Since $p\in H$ and $H$ was constructed not to contain $U$ or any of its boundary points, this is a contradiction. 

To prove the second part of the theorem, recall that Proposition \ref{prop:fails} shows that when $d\ge 2$ and $n = dk$, we can choose a\edit{n open} $k$-flexible sunflower $\U$ in $\R^d$ and points in each petal whose convex hull does not intersect the center of $\U$. This proves the result. 
\end{proof}

\begin{remark}
Note that when $k=1$, Theorem \ref{thm:flexible} is the same as Theorem \ref{thm:sunflower} (the usual Sunflower Theorem), and the application of Tverberg's theorem in the proof above reduces to an application of Radon's Theorem. Thus the fact that Theorem \ref{thm:flexible} generalizes Theorem \ref{thm:sunflower} is directly analogous to the fact that Tverberg's theorem generalizes Radon's. 
\end{remark}
\begin{remark}
In terms of neuroscientific motivation, flexible sunflowers are natural to investigate. Allowing some codewords beyond singletons, but of a fixed weight, accounts for some tolerance to error in data gathering and also captures a wider range of possibilities. We hope that flexible sunflowers may yield meaningful bounds on dimensions in experimental data. 
\end{remark}

Theorem \ref{thm:flexible} has implications regarding  the open embedding dimensions of intersection complete codes, which we will illustrate in Section \ref{subsec:SCoverD}, in particular by generalizing the families $\S_n$ and $\S_\Delta$ that were defined in Sections \ref{sec:sunflowercodeversion} and \ref{sec:SDelta} respectively. 

We conclude with a corollary which examines the extremal case in which we have a $k$-flexible sunflower $\U$ with $n=dk$ petals for which Theorem \ref{thm:flexible} fails. In this case Theorem \ref{thm:flexible} implies $\code(\U)$ must contain at least one codeword of weight $k$, but we can actually say something slightly stronger: 

\begin{corollary}\label{cor:manywords}
Let $\U = \{U_1,\ldots, U_n\}$ be an open $k$-flexible sunflower in $\R^d$. Suppose that $n = dk$, and there exist points $p_1,\ldots, p_n$ such that $p_i\in U_i$ and $\conv\{p_1,\ldots, p_n\}$ does not contain a point in the center of $\U$. Then $\code(\U)$ contains at least $d$ distinct codewords of weight $k$.
\end{corollary}
\begin{proof}
We work by induction on $k$. When $k=1$ the result is clear since if there are fewer than $d$ codewords of weight $k$ in $\code(\U)$ then some $U_i$ is equal to the center of $\U$, and so some $p_i$ lies in the center of $\U$, a contradiction. For $k\ge 2$, suppose for contradiction that $\code(\U)$ contains fewer than $d$ codewords of weight $k$. For each of these codewords $c$, select some petal $U_i$ with $i\in c$. Deleting these $U_i$ yields a $(k-1)$-flexible sunflower, and since we have deleted fewer than $d$ petals our new $(k-1)$-flexible sunflower has more than $d(k-1)$ petals. But the same choice of $p_i$ yields a collection of points whose convex hull does not contain a point in the center of this $(k-1)$-flexible sunflower, contradicting Theorem \ref{thm:flexible}.
\end{proof}

\section{Tangled Sunflowers}\label{sec:Tn}

For $n\ge 1$ we construct a \edit{family of} intersection complete code\edit{s} $\T_n\subseteq 2^{[2n]}$, and investigate \edit{the open embedding dimensions of these codes. The code $\T_n$ is constructed so that its realizations consist of two sunflowers with $n$ petals that are ``tangled" in the sense that the $i$-th petal of the first sunflower meets the $i$-th petal of the second, and no other incidences occur.} We use \edit{the sunflower theorem (}Theorem \ref{thm:sunflower}\edit{)} to prove the following: for \edit{any} $d\ge 1$ there exists $n$ such that \edit{$\odim(\T_n) = d$.} Thus for every $d\ge 1$, one of the $\T_n$ codes describes an arrangement of convex open sets which can be achieved in $\R^d$ but not a smaller dimension. Beyond this statement and some basic bounds, however, determining the \edit{exact open embedding dimension of $\T_n$} remains an open problem, ripe for future investigation.

\begin{definition}\label{def:Tn}
Let $n\ge 1$. Define $\T_n\subseteq 2^{[2n]}$ to be the code consisting of the following codewords:\begin{itemize}
\item[(i)] $\{2k-1, 2k\}$ for $k=1,2,\ldots, n$,
\item[(ii)] $\{1, 3,5,\ldots, 2n-1\}$ and $\{2,4,6,\ldots, 2n\}$,
\item[(iii)] all singletons, and
\item[(iv)] the empty set. 
\end{itemize}
For each $n$ define $t_n \od \odim( \T_n)$.
\end{definition}

Observe that codewords of type (i) and (ii) are the maximal codewords in $\T_n$ for $n\ge 2$; in particular $\T_n$ has $n+2$ maximal codewords. Furthermore observe that $\T_n$ is intersection complete, and hence \edit{open} convex. Thus $t_n$ is finite for all $n$. 

Moreover, note that the odd-numbered sets in any realization of $\T_n$ form an $n$-petal sunflower, as do the even-numbered sets. These two sunflowers are ``tangled," in that their petals are matched and \edit{the matched petals} overlap \edit{each other.}

\begin{example}\label{ex:S1-S4}
The first four $\T_n$ are given below: \begin{align*}
\T_1 &= \{\mathbf{12},1, 2, \emptyset\},\\
\T_2 &= \{\mathbf{13}, \mathbf{24}, \mathbf{12}, \mathbf{34}, 1, 2, 3, 4, \emptyset\},\\
\T_3 &= \{\mathbf{135}, \mathbf{246}, \mathbf{12}, \mathbf{34}, \mathbf{56}, 1, 2, 3, 4, 5, 6, \emptyset\}. \\
\T_4 &= \{\mathbf{1357}, \mathbf{2468}, \mathbf{12}, \mathbf{34}, \mathbf{56}, \mathbf{78}, 1, 2, 3, 4, 5, 6,7, 8, \emptyset\}. 
\end{align*}
These have \edit{open} convex realizations in $\R^1, \R^2$, $\R^3$, and $\R^3$ respectively.\[
\includegraphics[width=32em]{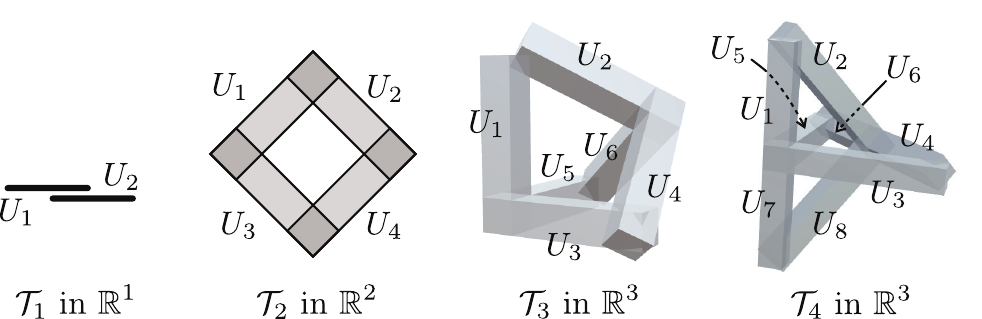}
\]
\end{example}

We will see that in fact each of the realizations in Example \ref{ex:S1-S4} is minimal with respect to dimension. That is, $t_1 = 1, t_2 =2, t_3 = t_4=3$. To build towards this result, we first prove some general results about the minimal embedding dimensions $\{t_n\mid n\ge 1\}$. 

\begin{proposition}\label{prop:sublinear}
For all $n\ge 1$, $t_n\le t_{n+1}\le t_n+1$. That is, the sequence $\{t_n\mid n\ge 1\}$ is weakly increasing and changes by at most 1 at each step. 
\end{proposition}
\begin{proof}
The inequality $t_n\le t_{n+1}$ follows from the fact that a realization of $\T_n$ can be obtained from a realization of $\T_{n+1}$ by simply deleting $U_{2n+1}$ and $U_{2n+2}$. To prove the inequality $t_{n+1}\le t_n + 1$ we argue that if $\T_n$ is \edit{open} convex in $\R^d$, then $\T_{n+1}$ is \edit{open} convex in $\R^{d+1}$.

Since $\T_n$ is intersection complete, we may apply Lemma \ref{lem:trimrealization} to obtain a\edit{n open convex} realization $\U = \{U_1,\ldots, U_{2n}\}$ of $\T_n$ in $\R^d$ in which disjoint $U_\sigma$ have positive distance between them. We will use this to create a\edit{n open convex} realization of $\T_{n+1}$ in $\R^{d+1}$. To start, identify $\R^d$ with the subspace of $\R^{d+1}$ in which $x_{d+1} = 0$, and define $W_1 = U_1\cap U_3\cap\cdots\cap U_{2n-1}$ and $W_2 = U_2\cap U_4\cap\cdots \cap U_{2n}$. We may assume that the origin lies in $W_1$. Now choose a vector $w\in W_2$ and a small positive $\varepsilon$, and define a collection $\V = \{V_1, V_2,\ldots, V_{2n+2}\}$ as follows:\[
V_i = \begin{cases}
\{v + \gamma e_{d+1} \mid v\in U_i \text{ and } 0 < \gamma < \varepsilon\} & \text{for } i = 1, 3, \ldots, 2n-1,\\
\{v + \gamma (e_{d+1}-w) \mid v\in U_i \text{ and } 0 < \gamma < \varepsilon\} & \text{for } i = 2, 4, \ldots, 2n,\\
\{v + \gamma e_{d+1} \mid v\in W_1\text{ and } \gamma >0\}& \text{for } i=2n+1,\\
\{v + \gamma (e_{d+1}-w) \mid v\in W_2\text{ and } \gamma >0\}& \text{for } i=2n+2.
\end{cases}
\]

This construction is shown below when $d=2$. The set $V_{2n+1}$ is a vertical \edit{prism} over $W_1$, and the set $V_{2n+2}$ is the skewed \edit{prism} over $W_2$. The remaining $V_i$ are $\varepsilon$-thick \edit{prisms} over the corresponding $U_i$, with even $V_i$ skewed at the same angle as $V_{2n+2}$. The origin is represented by the black dot. \[
\includegraphics[width=24em]{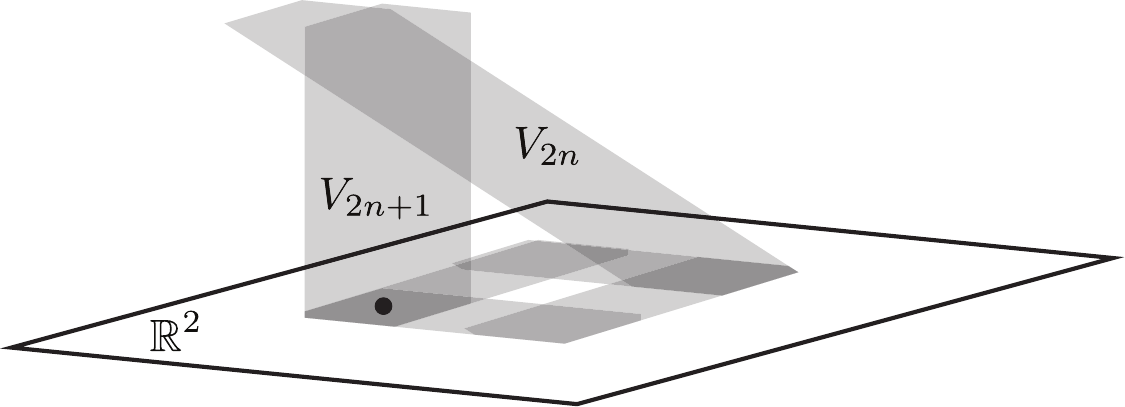}
\]

We claim that the collection $\V$ is a\edit{n open convex} realization of $\T_{n+1}$. First, observe that all $V_i$ are open and convex in $\R^{d+1}$. To see that they form a realization of $\T_{n+1}$, we must check that the odd $V_i$ and even $V_i$ both form sunflowers, and that only the appropriate petals intersect one another.

For the odd $V_i$, note that $\{V_1,V_3,\ldots, V_{2n-1}\}$ is a sunflower since the odd $U_i$ form a sunflower. Adding $V_{2n+1}$ to this collection preserves the sunflower property since $V_{2n+1}$ is simply the product of $W_1$ with an open ray. Similar logic holds for the even $V_i$: we see that $\{V_2,V_4,\ldots, V_{2n}\}$ forms a sunflower, and the additional petal $V_{2n+2}$ only overlaps any other petal in the region $\{v + \gamma (e_{d+1}-w) \mid v\in W_2\text{ and } 0<\gamma< \varepsilon\}$, which is the intersection of all the petals.

To see that the petals overlap in the correct manner, first note that $V_{2i-1}\cap V_{2i}$ is nonempty for $i=1,\ldots, n$ since the same holds for $U_{2i-1}\cap U_{2i}$. For $V_{2n+1}\cap V_{2n+2}$, simply note that $e_{d+1}\in V_{2n+1}\cap V_{2n+2}$ so the intersection is nonempty. Thus we have at least the appropriate overlapping between the petals of our two sunflowers, and it remains to show that no additional overlap has been introduced.

For this it suffices to argue that for all $j< k$ with different parity, the sets $V_j$ and $V_k$ are disjoint unless $j=2i-1$ and $k=2i$. We know that this property holds for the $U_i$, and since we chose a nondegenerate realization we know that disjoint $U_i$ have positive distance between them. Except for $V_{2n+1}$ and $V_{2n+2}$, all the $V_i$ are simply a slightly thickened $U_i$, possibly with a small skew by the vector $w$. By choosing $\varepsilon$ small enough, we can assume that the skew does not overcome the distance between disjoint $U_i$, so the $V_i$ satisfy the same disjointness for $i=1,2,\ldots, 2n$. This leaves the case of $V_{2n+1}$ and $V_{2n+2}$. For these, observe that all $V_i$ with $i\le 2n$ contain only points whose $(d+1)$-st coordinate is between 0 and $\varepsilon$. As discussed previously, the only points in $V_{2n+1}$ and $V_{2n+2}$ whose $(d+1)$-st coordinate satisfies these constraints are those in the center of the respective sunflowers. Thus neither of these sets overlap any petals they should not, and we have indeed formed a\edit{n open convex} realization of $\T_{n+1}$ in $\R^{d+1}$. This proves the result. 
\end{proof}

\begin{theorem}\label{thm:unbounded}
For all $n$, $t_n \ge \lceil n/2 \rceil$. In particular, the sequence $\{t_n\mid n\ge 1\}$ is unbounded.
\end{theorem}
\begin{proof}
Let $d = \lceil n/2 \rceil-1$. We must show that $\T_n$ does not have a\edit{n open convex} realization in $\R^d$. Suppose for contradiction that such a realization existed, consisting of sets $\{U_1,\ldots, U_{2n}\}$. Define $V_1 = U_1\cap U_3\cap \cdots\cap U_{2n-1}$ and $V_2 = U_2\cap U_4\cap\cdots\cap U_{2n}$. Observe that $V_1$ and $V_2$ are disjoint, \edit{nonempty, convex, and open. Thus we can choose a hyperplane $H$} separating $V_1$ and $V_2$.

Choose $p_1\in V_1$ and $p_2\in V_2$, and for $k\in[n]$ choose a point $q_k\in U_{2k-1}\cap U_{2k}$ (this intersection is nonempty since $\{2k-1, 2k\}$ is a codeword in $\T_n$). Now, for $k\in[n]$ consider the line segments $L_k = \overline{p_1q_k}$ and $M_k = \overline{q_k p_2}$. The union $L_k\cup M_k$ forms a path that begins on one side of $H$ and ends on the other, so for all $k$ either $L_k$ or $M_k$ contains a point in $H$ (and possibly both do). By choice of $d$ and pigeonhole principle, either at least $d+1$ of the line segments $\{L_k\}$ contain a point in $H$, or at least $d+1$ of the line segments $\{M_k\}$ contain a point in $H$. 

Without loss of generality, we may assume that at least $d+1$ of the $\{L_k\}$ contain a point $p_k$ in $H$. The convex hull of these $p_k$ lies in $H$, and therefore does not intersect the center $V_1$ of the sunflower $\{U_1,U_3,\ldots, U_{2n-1}\}$. But $L_k\subseteq U_{2k-1}$, so each $p_k$ lies in the petal $U_{2k-1}$. Since there are at least $d+1$ points $p_k$, Theorem \ref{thm:sunflower} implies that their convex hull \edit{(and thus $H$)} must intersect $V_1$, a contradiction.\end{proof}

\begin{corollary}
The sequence $\{t_n\mid n\ge 1\}$ takes on all positive integer values.
\end{corollary}
\begin{proof}
We know that $t_1 = 1$. Theorem \ref{thm:unbounded} implies that the sequence is unbounded, and Proposition \ref{prop:sublinear} tells us that it increases by at most 1 at each step. Thus it must achieve every positive integer value. 
\end{proof}

In the remainder of this section, we determine $t_n$ for all $n\le 5$. The arguments used below are concrete, but seem difficult to generalize.
\begin{proposition}\label{prop:T3}
The code $\T_3$ does not have a\edit{n open convex} realization in $\R^2$, but does have a\edit{n open convex} realization in $\R^3$. 
\end{proposition}
\begin{proof}
A\edit{n open convex} realization of $\T_3$ in $\R^3$ is given in Example \ref{ex:S1-S4}. Thus we just have to argue that $\T_3$ does not have a\edit{n open} convex realization in $\R^2$. Suppose for contradiction that $\{U_1, U_2, U_3, U_4, U_5, U_6\}$ is \edit{such} a realization of $\T_3$ in $\R^2$. Choose points $q_1\in U_1\cap U_2, q_2\in U_3\cap U_4$, and $q_3\in U_5\cap U_6$. Note that $\{U_1,U_3,U_5\}$ and $\{U_2, U_4,U_6\}$ are both sunflowers and that $\{q_1,q_2,q_3\}$ is a set containing one point from each petal for both of these sunflowers. By Theorem \ref{thm:sunflower} the triangle $\conv\{q_1,q_2,q_3\}$ contains a point $p_1\in U_1\cap U_3\cap U_5$ and $p_2\in U_2\cap U_4\cap U_6$. Since all the $U_i$ are open sets, we may assume that $\{p_1,p_2, q_1,q_2,q_3\}$ is in general position. The set of points $\{p_1, q_1,q_2,q_3\}$ can be visualized as follows:\[
\includegraphics[width=8em]{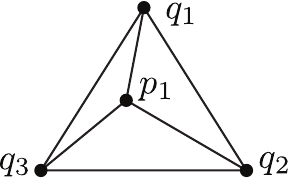}
\]
Now, $p_2$ falls in the interior of one of the three triangular regions surrounding $p_1$. Suppose that $p_2$ lies in the interior of $\conv\{p_1, q_1, q_2\}$ (i.e. the top right triangle above). Then consider the line segment $L = \overline{p_2 q_3}$, observing that $L$ is contained in $U_6$. The line segment $L$ must cross either the line segment $\overline{p_1q_1}\subseteq U_1$ or $\overline{p_1q_2}\subseteq U_3$. In the former case we see  that $U_6\cap U_1 \neq \emptyset$, and in the latter $U_6\cap U_3 \neq\emptyset$. But there is no codeword in $\T_3$ containing $\{1,6\}$ or $\{3,6\}$, so both of these situations lead to a contradiction. Thus $\T_3$ is not convex in $\R^2$. 
\end{proof}

The lemma below will allow us to prove that $t_5 \ge 4$ by showing that if $\T_5$ has a\edit{n open convex} realization in $\R^3$, then $\T_3$ has a\edit{n open convex} realization in $\R^2$, contradicting Proposition \ref{prop:T3}.

\begin{lemma}\label{lem:5points}
Given five points in $\R^3$ in general position, there exists a plane $H$ containing three of the points and with the remaining two points on opposite sides of $H$. 
\end{lemma}
\begin{proof}
Up to affine transformation we may assume that our set of points is $\{0, e_1, e_2, e_3, p\}$ where $p$ is a point none of whose coordinates are zero. We consider two cases. First suppose that one of the coordinates of $p$ is negative. By permuting our coordinates we can assume this is the last coordinate. Then choose $H = \spann\{e_1,e_2\}$. This contains the three points $0, e_1$, and $e_2$. Moreover since $e_3$ has positive last coordinate and $p$ has negative last coordinate, they lie on opposite sides of $H$ and the lemma follows.

Otherwise every coordinate of $p$ is positive. In this case, write $p=(x,y,z)$ and choose $H = \spann\{e_3, p\}$. Observe that $H$ contains the three points $0, e_3,$ and $p$, and that $v = (y,-x,0)$ is a normal vector to $H$. We see that $v\cdot e_1 >0$ and $v\cdot e_2< 0$, so the remaining two points $e_1$ and $e_2$ lie on opposite sides of $H$. This proves the result. 
\end{proof}

\begin{proposition}\label{prop:S5}
The code $\T_5$ does not have a\edit{n open convex} realization in $\R^3$. 
\end{proposition}
\begin{proof}
Suppose for contradiction that we have a\edit{n open convex} realization $\U = \{U_1, U_2,\ldots, U_{10}\}$ of $\T_5$ in $\R^3$. For $i=1,\ldots, 5$, choose a point $p_i$ in the open set $U_{2i-1}\cap U_{2i}$, such that all $p_i$ are in general position. Applying Lemma \ref{lem:5points} to these five points, we obtain a hyperplane $H$ \edit{which} contains three of them, and with the remaining two on opposite sides. By permuting the labels on our realization of $\T_5$, we may assume that $p_1, p_2,$ and $p_3$ all lie in $H$.

Now, consider the two tetrahedra $\Delta_1 = \conv\{p_1,p_2,p_3,p_4\}$ and $\Delta_2 = \conv\{p_1,p_2,p_3,p_5\}$. The vertices of these tetrahedra belong to distinct petals of the sunflowers $\{U_1, U_3,\ldots, U_9\}$ and $\{U_2, U_4, \ldots,U_{10}\}$, so by Theorem \ref{thm:sunflower} each of these tetrahedra contain a point in the center of both of these sunflowers. Since the tetrahedra lie on opposite sides of $H$, each of the centers of these two sunflowers contains a point on each side of $H$. But the center of a sunflower is convex, and so $H$ itself must contain a point in the center of each of the two sunflowers.

With this observation, consider the set $\V = \{V_1,\ldots, V_6\}$ where $V_i = U_i\cap H$. Since $H\cong \R^2$, we can regard this set as a\edit{n open} convex realization of a code in $\R^2$. We claim that in fact this code is $\T_3$. To verify this, it suffices to show that (i) $\{V_1, V_3, V_5\}$ and $\{V_2, V_4, V_6\}$ are both sunflowers and (ii) that $V_1\cap V_2$, $V_3\cap V_4$, and $V_5\cap V_6$ are nonempty, and that (iii) no other petals overlap. 

Condition (i) follows from the fact that the $V_i$ are subsets of the $U_i$ and that the sunflowers making up the realization of $\T_5$ both have centers that intersect $H$. Condition (ii) follows by considering the points $p_1, p_2$, and $p_3$, which all lie in the desired respective intersections.
Condition (iii) is a consequence of the fact that the the petals of the $U_i$ sunflowers overlap appropriately. 

However, this is a contradiction: $\T_3$ is not convex in $\R^2$ by Proposition \ref{prop:T3}. Thus $\T_5$ cannot be convex in $\R^3$. 
\end{proof}

\begin{corollary}\label{cor:smallvalues}
The sequence $t_n$ begins as follows:\[\begin{tabular}{|c||c|c|c|c|c|}
\hline $n$ & $1$ & $2$ & $3$ & $4$ & $5$\\\hline
$t_n$ & $1$ & $2$ & $3$ & $3$ & $4$\\\hline
\end{tabular}\]
\end{corollary}
\begin{proof}
Clearly $t_1 = 1$ since $\T_1$ is convex in $\R^1$ but has more than one codeword, so is not convex in $\R^0$. The code $\T_2$ has a realization in $\R^2$ as given in Example \ref{ex:S1-S4}, but has no realization in $\R^1$ since any realization contains a non-crossing loop. Thus $t_2 = 2$.

Note that $t_3 \le 3$ and $t_4\le 3$ by Example \ref{ex:S1-S4}, and both bounds are tight by Proposition \ref{prop:T3} and monotonicity of the $t_n$. By Proposition \ref{prop:S5} we know that $t_5\ge 4$, and simultaneously Proposition \ref{prop:sublinear} implies that $t_5 \le t_4+1 = 4$. This proves the result. 
\end{proof}

The proofs presented in Propositions \ref{prop:T3} and \ref{prop:S5} are both somewhat ad hoc and do not seem ripe for generalization. Determining $t_n$ for $n\ge 6$ remains an open problem, perhaps of significant difficulty.

\section{Contextualizing Our Results via Code Minors}\label{sec:pcode}

\edit{In this section we will situate our results in the framework of code morphisms and minors. We begin in Section \ref{subsec:minors} by recalling several basic definitions and results regarding morphisms and minors, following the most recent treatment of this material which appears in \cite[Chapters 3 and 4]{amziphdthesis}. In Section \ref{subsec:SnSDeltaTn} we explain how the families of codes $\S_n$, $\S_\Delta$, and $\T_n$ relate to one another in the framework of minors. We then generalize the codes $\S_n$ and $\S_\Delta$ as well as their accompanying results in Section \ref{subsec:SCoverD}.}

\subsection{Minors of Codes}\label{subsec:minors}

In \cite{morphisms}, we introduced a \edit{framework} of morphism\edit{s} for neural codes. \edit{Morphisms allow us to define a notion of ``minors" for codes, analogous to (but not a generalization of) minors of graphs or matroids.} Morphisms \edit{and minors} have a strong relationship to convexity \edit{(see Theorem \ref{thm:minorclosed} below)}, and provide a useful \edit{context} in which to state and compare results about convex neural codes. 

\edit{For the definition below, recall from Definition \ref{def:trunk} that a trunk in a code $\C$ is a set of the form $\Tk_\C(\sigma) \od \{c\in \C\mid \sigma\subseteq c\}$ for some $\sigma\subseteq [n]$.}

\begin{definition}\label{def:morphism}
Let $\C$ and $\D$ be codes. A function $f:\C\to\D$ is called a \emph{morphism} if the preimage of any \edit{proper} trunk in $\D$ under $f$ is a \edit{proper} trunk in $\C$. \edit{An \emph{isomorphism} is a morphism with an inverse function that is also a morphism.}
\end{definition}

\begin{definition}
We say that a code $\D$ is a \emph{minor} of a code $\C$ if there exists a \edit{surjective morphism $f:\C\to \D$.}
\end{definition}

\edit{The relation ``$\D$ is a minor of $\C$" forms a partial order on isomorphism classes of codes. We denote the resulting partially ordered set by $\ParCode$, and write $\D\le \C$ when $\D$ is a minor of $\C$. Note that morphisms and minors are defined in a purely combinatorial manner. However, minors can be used to understand geometric information about realizations of codes, as the following theorem indicates.}

\begin{theorem}\label{thm:minorclosed}
The following properties are \emph{minor-closed} (that is, if $\C$ has one of the properties below, then so does every $\D\le \C$):\begin{itemize}
\item Open convexity in $\R^d$,
\item Closed convexity in $\R^d$,
\item Non-degenerate open/closed convexity in $\R^d$ (see Definition \ref{def:nondegen}),
\item Intersection completeness.
\end{itemize}
\end{theorem}

\edit{Theorem \ref{thm:minorclosed} follows from a more general observation, first noted by Caitlin Lienkaemper: if $\C\subseteq 2^{[n]}$ is a code with a (possibly not convex or open) realization $\U = \{U_1,\ldots, U_n\}$ in a space $X$, then there is a bijection}
\[
\left\{\parbox{6em}{\begin{center}\edit{Minors of $\C$}\end{center}}\right\} \quad\longleftrightarrow \quad\left\{\parbox{9em}{\begin{center}\edit{Codes that can be\\realized in $X$ using sets of the form $U_\sigma$}\end{center}}\right\}. 
\]
\edit{For details on this result, see \cite[Section 4.2]{amziphdthesis} and \cite[Section 4]{matroids}.}

\edit{Theorem \ref{thm:minorclosed}} implies that intersection completeness is an isomorphism invariant, and that restricting \edit{our attention} to intersection complete codes amounts to restricting to a \edit{minor-closed family} in $\ParCode$. Throughout the rest of this section, we will examine exclusively intersection complete codes, with the partial order inherited from $\ParCode$. We will focus on the open embedding dimensions of these codes.

One can visualize \edit{minors as stratifying intersection complete codes} into different ``layers" \edit{according to their open embedding dimensions,} as \edit{sketched} in the figure below. \edit{In the figure, ``Open convex in $\R^d$" simply refers to codes whose open embedding dimension is equal to $d$.}\[
\includegraphics[width=30em]{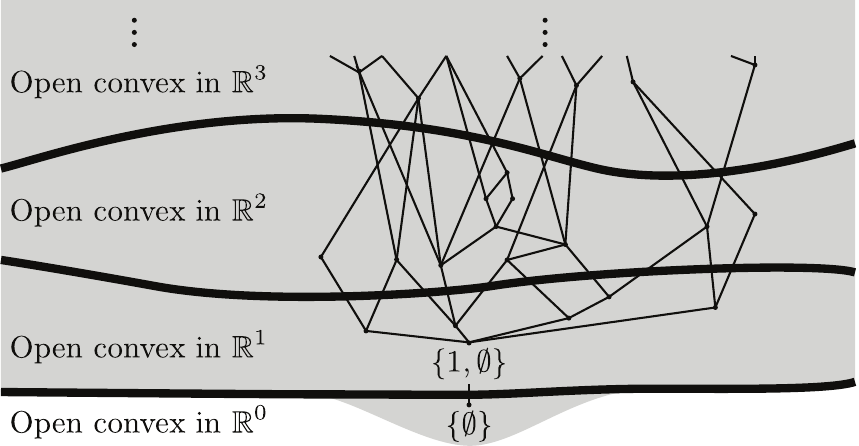}
\]
Note that $\{\emptyset\}$ is the only intersection complete code whose open embedding dimension is zero. For any $d\ge1$, however, there are infinitely many intersection complete codes with open embedding dimension $d$. 

The figure above is slightly misleading: each ``layer" of codes with open embedding dimension $d$ is not finitely thick. Indeed, each layer may contain chains that are infinitely long (``tall"), and antichains that are infinitely large (``wide"). 

\edit{Our aim in Section \ref{subsec:SnSDeltaTn}} will be to understand where the codes we have constructed in this paper sit inside this partial order. In Section \ref{subsec:SCoverD}, we will provide some more general examples \edit{and results} using Theorem \ref{thm:flexible}.

We will make heavy use of the following definition and proposition, which give a combinatorial description of all morphisms. For details, see \cite[Section 3.2]{amziphdthesis}.

\begin{definition}\label{def:determined}
Let $\C\subseteq 2^{[n]}$ be a code, and for $i\in[m]$ let $T_i\subseteq \C$ be a \edit{proper} trunk. The \emph{morphism determined by the trunks $\{T_1,\ldots, T_m\}$} is the map $f:\C\to 2^{[m]}$ given by $f(c) = \{i\in[m] \mid c\in T_i\}$.
\end{definition}

\begin{proposition}\label{prop:determined}
The map described in Definition \ref{def:determined} is a morphism from $\C$ to $2^{[m]}$. Moreover, every morphism arises in this way. Formally, for codes $\C\subseteq 2^{[n]}$ and $\D\subseteq 2^{[m]}$, and any morphism $f:\C\to \D$, $f$ is the morphism determined by the trunks $\{T_i\od f^{-1}(\Tk_\D(i))\mid i\in[m]\}$ (restricted to the range $\D$). Equivalently, for all $c\in \C$,\[
f(c) = \{i\in[m]\mid c\in f^{-1}(\Tk_\D(i))\}.
\]
\end{proposition}

\subsection{The codes $\S_n, \S_\Delta,$ and $\T_n$ in $\ParCode$}\label{subsec:SnSDeltaTn}

Let us begin by establishing a relationship between codes of the type $\S_n$ and the type $\S_\Delta$. Recall that $\S_n$ is a special case of $\S_\Delta$---in particular, $\S_n = \S_\Delta$ where $\Delta$ is $n$ points. More generally, we have the following:

\begin{proposition}\label{prop:SDeltaSn} Let $\Delta\subseteq 2^{[n]}$ be a simplicial complex with $m$ facets. Then there exists a surjective morphism $\S_\Delta\to \S_{m}$. In particular, $\S_{m} \le \S_\Delta$.  
\end{proposition}
\begin{proof}
Let $F_1,\ldots, F_m$ be the facets of $\Delta$. For $i\in[m]$ define $T_i = \Tk_{\S_\Delta}(F_i)$, and define $T_{m+1} = \Tk_{\S_\Delta}(n+1)$. We claim that $\S_m$ is the image of $\S_{\Delta}$ under the morphism $f$ determined by the trunks $\{T_1,T_2,\ldots, T_{m+1}\}$. Recall from Definition \ref{def:SDelta} that the codewords of $\S_\Delta$ are:\begin{itemize}
\item $\sigma$ for $\sigma\in\Delta$,
\item $\sigma \cup \{n+1\}$ for $\sigma\in \Delta$, and
\item $[n]$. 
\end{itemize}
The images of these codewords under $f$ are as follows:\begin{itemize}
\item $f(\sigma)$ is equal to $\emptyset$ if $\sigma$ is not a facet of $\Delta$, and equal to $\{i\}$ if $\sigma = F_i$,
\item $f(\sigma\cup\{n+1\})$ is equal to $\{m+1\}$ if $\sigma$ is not a facet of $\Delta$, and equal to $\{i, m+1\}$ if $\sigma = F_i$, and
\item $f([n]) = [m]$ since $[n]$ contains all facets of $\Delta$, but does not contain $m+1$.
\end{itemize}
But comparing these \edit{images} to Definition \ref{def:Sn}, we see that these are exactly the codewords of $\S_m$, proving the result.
\end{proof}

\begin{remark}\label{rem:Sm} One way to think of Proposition \ref{prop:SDeltaSn} is as follows. The set \[\{\S_\Delta\mid \Delta\text{ is a simplicial complex with $m$ facets}\}\] inherits a partial order from $\ParCode$, and with this inherited order $\S_m$ is the unique minimal element of the set. Theorem \ref{thm:SDelta} says that for $m\ge 2$ all of these live in the ``layer" of codes with open embedding dimension $m$. We can visualize this situation as follows. \[
\includegraphics[width=18em]{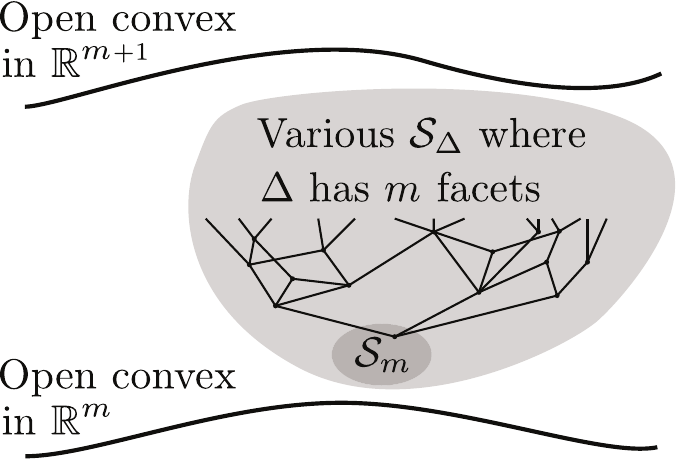}
\]\end{remark}

It is also worth noting the following, regarding the codes $\T_n$ described in Section \ref{sec:Tn}.
\begin{proposition}\label{prop:Tnchain}
For any $n\ge 1$, $\T_n\le \T_{n+1}$. In particular, the codes $\{\T_n\mid n\ge 1\}$ form a chain in $\ParCode$.
\end{proposition}
\begin{proof}
Given a set of neurons $\sigma\subseteq [n]$, one can form a natural ``restriction" of a code $\C\subseteq 2^{[n]}$ by mapping $c\mapsto c\cap \sigma$.  This restriction is a morphism (see \cite[Section 2]{morphisms}). In the case of the codes $\T_n\subseteq 2^{[2n]}$, one can note that $\T_n$ is the image of $\T_{n+1}$ under the restriction map with $\sigma = [2n]\subseteq [2n+2]$. This surjective morphism from $\T_{n+1}$ to $\T_n$ implies that $\T_n\le \T_{n+1}$ as desired.
\end{proof}

\subsection{Generalizing $\S_n$ and $\S_\Delta$ using Theorem \ref{thm:flexible}}\label{subsec:SCoverD}

We begin with a definition generalizing that of $\S_\Delta$. \begin{definition}\label{def:SCoverD}
Let $\D\subseteq \C\subseteq 2^{[n]}$ be intersection complete codes. We define \[
\S_{\C/\D} \od \C \cup \{[n]\}\cup \{d\cup\{n+1\}\mid d\in \D\} \subseteq 2^{[n+1]}.
\] 
Note that choosing $\D = \{\text{minimal nonempty codewords in $\C$}\}$\edit{$\cup \{\emptyset\}$} always satisfies the above conditions. In this case, we will let $\S_{\C/\text{min}}$ denote $\S_{\C/\D}$. 
\end{definition}

Qualitatively, $\S_{\C/\D}$ is the result of forming a flexible sunflower using the codewords in $\C$, and then ``gluing" the petals of that sunflower to a new set $U_{n+1}$ along codewords in $\D$. Observe that $\S_\Delta$ of Definition \ref{def:SDelta} is equal to $\S_{\Delta/\Delta}$ in this notation. Also, if $\C = \{\{1\},\{2\},\ldots, \{n\}, \emptyset\}$, then we see $\S_n$ of Definition \ref{def:Sn} is equal to $\S_{\C/\text{min}}$. 

\begin{proposition}
Let $\D\subseteq \C\subseteq 2^{[n]}$ be intersection complete codes.  The code $\S_{\C/\D}$ is intersection complete. If $\D$ has $m$ maximal codewords and does not contain $[n]$, then $\S_{\C/\D}$ has $m+1$ maximal codewords. In particular, $\odim(\S_{\C/\D}) \le \max\{2,m\}$. 
\end{proposition}
\begin{proof}
Codewords in $\S_{\C/\D}$ come in three types: codewords from $\C$, the codeword $[n]$, and those of the form $d\cup\{n+1\}$ where $d\in\D$. Since $\C$ and $\D$ are intersection complete, the intersection of two codewords of the same type always yields another codeword of that type (and hence lying in $\S_{\C/\D}$). This leaves the intersections of codewords of different types. The intersection of a codeword in $\C$ with $[n]$ is simply the same codeword in $\C$. The intersection of $d\cup\{n+1\}$ with $[n]$ is just $d$, which lies in $\S_{\C/\D}$ since $\D\subseteq \C$. Finally, the intersection of $c\in\C$ with $d\cup\{n+1\}$ is $c\cap d$, which lies in $\C$ since since $\C$ is intersection complete.

For the second part of the statement, note that if $d$ is a maximal codeword of $\D$, then $d\cup\{n+1\}$ is a maximal codeword of $\S_{\C/\D}$. Since $[n]\notin\D$, the codeword $[n]$ is also a maximal codeword of $\S_{\C/\D}$, yielding $m+1$ total maximal codewords. The bound on $\odim(\S_{\C/\D})$ follows immediately from \cite[Theorem 1.2]{openclosed}.
\end{proof}

The following proposition provides a generalization of Theorem \ref{thm:sunflowercodeversion} to the codes $\S_{\C/\text{min}}$.

\begin{proposition}\label{prop:SCovern}
Let $\C\subseteq 2^{[n]}$ be an intersection complete code \edit{that} contains every singleton set. Then \[\odim(\S_{\C/\text{min}}) \ge \left\lceil \frac{n}{\dim(\Delta(\C))+1}\right\rceil.\] 
\end{proposition}
\begin{proof}
We start with a degenerate case: if $n=1$, then $\C = \{\emptyset, 1\}$ and $\S_{\C/\text{min}} = \{12, 1, 2, \emptyset\}$. In this case $\odim(\S_{\C/\text{min}}) = 1$, while $n = 1$ and $\dim(\Delta(\C)) + 1 = 1$. We see that the bound given above is satisfied as desired.

Otherwise, $n\ge 2$. In this case, let $\{U_1,\ldots, U_{m+1}\}$ be an open convex realization of $\S_{\C/\text{min}}$ in $\R^d$. Since the minimal nonempty codewords of $\C$ are all singletons, the code $\S_{\C/\text{min}}$ consists of codewords from $\C$, the codeword $[n]$, codewords of the form $\{i, n+1\}$ where $i\in[n]$, and lastly the codeword $\{n+1\}$. Since $[n]$ is a codeword, the sets $\{U_1,\ldots, U_n\}$ all meet in a central point. In particular, $\{U_1,\ldots, U_n\}$ is a $k$-flexible sunflower, where $k$ is the largest weight of a codeword in $\C$ other than possibly $[n]$. In particular $k \le \dim(\Delta(\C))+1$, with equality if $[n]\notin \C$. 

 But consider the set $U_{n+1}$. This set does not meet $U_{[n]}$ since $[n+1]$ is not a codeword of $\S_{\C/\text{min}}$. However, it does touch each $U_i$ since $\{i, n+1\}$ is a codeword. If we choose $p_i\in U_i\cap U_{n+1}$, then the convex hull of $\{p_1,\ldots, p_n\}$ is contained in $U_{n+1}$ and therefore does not contain a point in the center of $\{U_1,\ldots, U_n\}$. By Theorem \ref{thm:flexible}, such a sampling of $p_i$ cannot be chosen if $n\ge dk+1$. Therefore we must have $n\le dk$. Rearranging, this implies $d\ge \lceil n/k\rceil$. Using the inequality $k\le \dim(\Delta(\C))+1$ yields the result.
\end{proof}

The added assumption in Proposition \ref{prop:SCovern} that $\C$ contains all singletons is not too restrictive, since adding singletons to an intersection complete code always maintains intersection completeness. 

Continuing our pattern of generalizations, the proposition below is analogous to Theorem \ref{thm:SDelta} and its second part generalizes Proposition \ref{prop:SDeltaSn}. 

\begin{proposition}\label{prop:SCoverDSE}
Let $\D\subseteq \C\subseteq 2^{[n]}$ be intersection complete codes.  Let $m\ge 2$ be the number of maximal codewords in $\D$, and let $k$ be the largest number of maximal codewords in $\D$ whose union lies in $\Delta(\C)$. Then there exists an intersection complete code $\E\subseteq 2^{[m]}$ containing all singleton sets such that (i) $k=\dim(\Delta(\E))+1$, and (ii) there exists a surjective morphism $\S_{\C/\D}\to \S_{\E/\text{min}}$. In particular,  $\S_{\E/\text{min}} \le \S_{\C/\D}$ and $m\ge \odim(\S_{\C/\D}) \ge \left\lceil \frac{m}{k}\right\rceil$.
\end{proposition}\begin{proof}
We will mirror the proof of Proposition \ref{prop:SDeltaSn}. Let $F_1,\ldots, F_m$ be the maximal codewords of $\D$. For $i\in[m]$ define $T_i = \Tk_{\S_{\C/\D}}(F_i)$, and define $T_{m+1} = \Tk_{\S_{\C/\D}}(n+1)$. Let us consider the image of $\S_{\C/\D}$ under the morphism $f$ determined by $\{T_1,\ldots, T_{m+1}\}$.  Recall from Definition \ref{def:SCoverD} that the codewords of $\S_{\C/\D}$ come in the following types:\begin{itemize}
\item $c$ for $c\in \C$,
\item $d\cup \{n+1\}$ for $d\in \D$, and
\item $[n]$.
\end{itemize} The images of these codewords under $f$ are as follows:\begin{itemize}
\item $f(c)$ is equal to $\{i\in[m]\mid c \text{ contains $F_i$}\}$,
\item $f(d\cup \{n+1\})$ is equal to $\{m+1\}$ if $d$ is not equal to some $F_i$, and is equal to $\{i,m+1\}$ if $d = F_i$,
\item $f([n]) = [m]$ since $[n]$ contains all maximal codewords in $\D$, but not $n+1$. 
\end{itemize}
Let $\E\subseteq 2^{[m]}$ be the collection of codewords in the first bullet above, i.e. $\E$ is the image of $\C$ under $f$. Since the image of an intersection complete code is again intersection complete, we see that $\E$ is intersection complete. Moreover, $\E$ contains every singleton set since $f(F_i) = \{i\}$. 

The image of $\S_{\C/\D}$ under $f$ therefore contains codewords in $\E$, codewords of the form $\{i, m+1\}$ for all $i\in[m]$, the codeword $\{m+1\}$, and $[m]$. But these are exactly the codewords of $\S_{\E/\text{min}}$. Thus $\S_{\E/\text{min}}$ is the image of $\S_{\C/\D}$ under $f$. 

To prove the result, it remains to show that $k = \dim(\Delta(\E)) + 1$. The codewords in $\E$ are of the form $f(c) = \{i\in[m]\mid F_i\subseteq c\}$. Thus a codeword in $\E$ corresponds to a collection of maximal codewords in $\D$ all of which are contained in some $c\in \C$. A codeword in $\E$ with largest weight thus corresponds to a largest possible collection of maximal codewords in $\D$ whose union is contained in $\Delta(\C)$. The largest such collection has size $k$ by definition, so any largest codeword in $\E$ has weight $k$, proving the result. 
\end{proof}

\begin{remark}Generalizing Remark \ref{rem:Sm} from the last section, we see that among all codes of the form $\S_{\C/\D}$ with parameters $m$ and $k$ as described in Proposition \ref{prop:SCoverDSE}, the minimal elements are always of the form $\S_{\E/\text{min}}$ where $\E\subseteq 2^{[m]}$ contains all singletons, and $k =\dim(\Delta(\E))+1$. The following diagram shows this:\[
\includegraphics[width=30em]{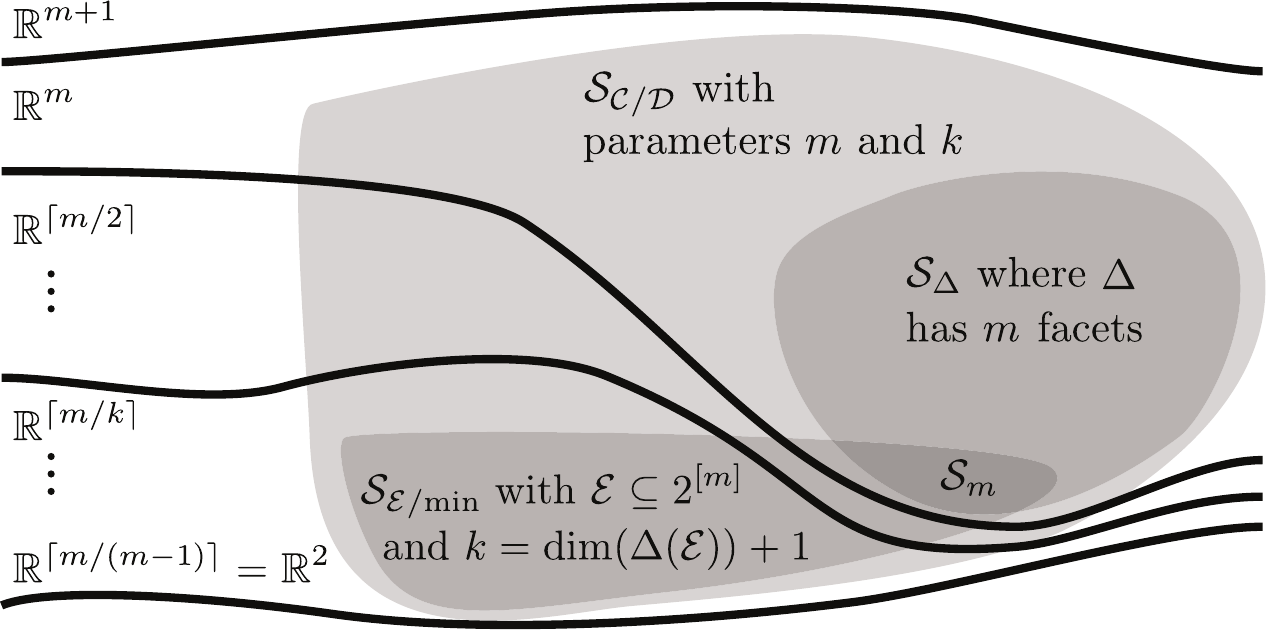}
\]
\end{remark}

These results use Theorem \ref{thm:flexible} to provide a more complete picture of the open embedding dimensions of intersection complete codes. There is still much to be done, however. As one example, the bound $m\ge \odim(\S_{\C/\D}) \ge \left\lceil \frac{m}{k}\right\rceil$ of Proposition \ref{prop:SCoverDSE} leaves quite a large gap for $k\ge 2$. Sharpening this bound based on the combinatorial structure of $\C$ and $\D$ would be a natural task of interest.

\section{Conclusion}\label{sec:conclusion}

We have seen a number of phenomena arise in the closed and open embedding dimensions of intersection complete codes. Some of these, like Theorems \ref{thm:complexes}, \ref{thm:cdimleodim}, and \ref{thm:cdimlinear}, gave us improved control over the embedding dimensions. Others, like Theorem \ref{thm:SDelta}, showed that embedding dimension may be difficult to control. With Theorem \ref{thm:flexible}, we developed new tools to understand open embedding dimension using $k$-flexible sunflowers, but the picture is still far from complete.

One direction for future work would be to search for analogous phenomena among codes that are not intersection complete. One could start with the following. 

\begin{question}\label{q:odimlecdim}
Does there exist a code $\C$ with $\odim(\C)<\cdim(\C)<\infty$?
\end{question}

Theorem \ref{thm:cdimleodim} tells us that such a code cannot be intersection complete. There are examples due to \cite{5neurons, openclosed} of codes with $\odim(\C) < \cdim(\C) = \infty$. \edit{These examples rely on a compactness argument to construct contradictory line segments of minimum distance in a hypothetical closed realization, proving $\cdim(\C) = \infty$.} A similar approach, paired with a classic convexity theorem that depends on dimension such as Radon's theorem, could yield a positive answer to Question \ref{q:odimlecdim}, and also possibly Question \ref{q:cdimexp} below. 

\begin{question}\label{q:cdimexp}
Little is known about whether closed embedding dimension can be large relative to the number of neurons, $n$. A few open areas to investigate are the following, in increasing order of difficulty:\begin{itemize}
\item Does there exist a code $\C\subseteq 2^{[n]}$ for which $\cdim(\C)$ is finite, but larger than $n-1$?
\item Does there exist a family of codes $\{\C_n\subseteq 2^{[n]}\mid n\ge 1\}$ such that $\cdim(\C_n)$ grows faster than any linear function of $n$?
\item Does there exist a family of codes $\{\C_n\subseteq 2^{[n]}\mid n\ge 1\}$ such that $\cdim(\C_n)$ grows faster than any polynomial function of $n$?
\end{itemize}  
Note that Theorem \ref{thm:cdimlinear} tells us that if such codes exist, they cannot be intersection complete. \edit{We have provided an} affirmative answers to the $\odim$ versions of the above questions \edit{via the family $\S_\Delta$, and in particular Corollary \ref{cor:exponential}. }
\end{question}



Regarding the tangled sunflower codes $\T_n$ of Section \ref{sec:Tn}, there is much to be done. A good first step would be to improve the embedding dimension bounds that we currently have, or, more ambitiously, find an exact characterization of the embedding dimension. 

\begin{question}
Does there exist an explicit characterization of the open embedding dimensions $t_n$ described in Definition \ref{def:Tn}? Can we improve the bounds of $\lceil n/2\rceil \le t_n \le n$? 
\end{question}

One might also consider codes that describe more than two sunflowers whose petals are ``tangled" (i.e. incident) in some way. This would be a significantly more complicated problem, but perhaps of some interest. Another generalization would be to consider a notion of tangled flexible sunflowers. This would be even more challenging to investigate, but would perhaps be more relevant to applications in experimental data. 

\begin{question}\label{q:simplicialabove}
In Section \ref{sec:pcode} we contextualized our results via a partial order on codes, denoted \edit{by} $\ParCode$. In this partial order, both $\odim$ and $\cdim$ are monotone functions. In \cite{morphisms} we showed that a code is intersection complete if and only if it lies below a simplicial complex in $\ParCode$. An interesting question is thus the following: do the simplicial complexes lying above an intersection complete code $\C$ in $\ParCode$ determine $\odim(\C)$? That is, among the simplicial complexes lying above $\C$ in $\ParCode$, does one have minimal embedding dimension equal to $\odim(\C)$?
\end{question}
A positive answer to the above question would reduce the problem of determining open embedding dimension for intersection complete codes to the problem of determining open embedding dimension for simplicial complexes, which is very closely tied to the well-studied problem of determining when a complex is $d$-representable, as described in \cite[Section 1.2]{tancer}. 

Note that the answer to \edit{Question \ref{q:simplicialabove}} cannot be positive when we replace $\odim$ with $\cdim$. Open and closed embedding dimension for simplicial complexes are always the same \edit{(recall Theorem \ref{thm:complexes})}, but the code $\S_3$ already shows that closed dimension and open dimension are different for intersection complete codes. \edit{In particular, $\cdim(\S_3) = 2$ while $\odim(\S_3) = 3$.} Thus the simplicial complexes \edit{that $\S_3$ is a minor of do not determine its closed embedding dimension, at least not simply as the minimum of their closed embedding dimensions, which must be at least 3.} 

In \edit{the proof of Theorem \ref{thm:SDelta} (}see Section \ref{sec:SDelta}\edit{)}, we showed that $\odim(\S_\Delta)$ was equal to the number of facets in $\Delta$ by showing that any realization of $\S_\Delta$ \edit{gives} rise to a realization of the code $\S_m$ described in Section \ref{sec:sunflowercodeversion}. \edit{Later, in Proposition \ref{prop:SDeltaSn}, we translated this argument to the context of minors, demonstrating a surjective morphism from $\S_\Delta$ to $\S_m$.} This technique could be generalized to analyze arbitrary codes as follows. Given a code $\C$, look for the largest $m$ so that there is a surjective morphism $\C\to \S_m$ \edit{(i.e. the largest $m$ so that $\S_m$ is a minor of $\C$)}. This largest $m$ then provides a lower bound on the open embedding dimension of $\C$. 

\edit{Importantly, this method is distinct from} existing techniques for providing lower bounds on $\odim(\C)$\edit{, which} rely on homological information obtained from $\Delta(\C)$ (see for example \cite{leray}). In contrast, the sunflower approach is completely agnostic to homology of $\Delta(\C)$. Whether \edit{sunflowers of convex open sets} could be useful in analyzing experimental data may be an interesting open question. \edit{As a start, it would be useful to determine whether searching for specific minors of a given code can be done in a computationally feasible manner.}

\begin{question}
Given a code $\C\subseteq 2^{[n]}$, is there an efficient algorithm to determine the largest $m$ so that $\S_m$ is a minor of $\C$? More generally, for what pairs of codes $\C\subseteq 2^{[n]}$ and $\D\subseteq 2^{[m]}$ can one efficiently recognize whether or not $\D$ is a minor of $\C$?
\end{question}

\section*{Acknowledgements}
 
 We would like to thank Florian Frick for raising the question of whether there exist open convex codes $\C\subseteq 2^{[n]}$ with $\odim(\C) > n-1$, and for interesting discussions on this question. Isabella Novik provided detailed feedback on initial drafts of this paper, as well as helpful discussions on its mathematical content. Anne Shiu also provided \edit{thorough} feedback on an initial draft of the paper. \edit{Finally, we would like to thank the anonymous referees for their comments, which greatly improved the presentation of our results. }

\bibliographystyle{plain}
\bibliography{neuralcodereferences}

\end{document}